\documentclass[12pt]{amsart}

\usepackage{amsmath, amssymb, amsfonts, amsthm, enumerate,graphicx}

\usepackage[all, cmtip]{xy}

\usepackage{hyperref}

\DeclareMathSymbol{\R}{\mathbin}{AMSb}{"52}
\DeclareMathSymbol{\N}{\mathbin}{AMSb}{"4E}
\DeclareMathSymbol{\Z}{\mathbin}{AMSb}{"5A}
\DeclareMathSymbol{\Q}{\mathbin}{AMSb}{"51}
\DeclareMathSymbol{\C}{\mathbin}{AMSb}{"43}

\newtheorem{theorem}{Theorem}[section]
\newtheorem{lemma}[theorem]{Lemma}
\newtheorem{proposition}[theorem]{Proposition}
\newtheorem{corollary}[theorem]{Corollary}

\DeclareMathOperator{\SL}{SL}
\DeclareMathOperator{\GL}{GL}
\DeclareMathOperator{\PSL}{PSL}
\DeclareMathOperator{\PGL}{PGL}
\DeclareMathOperator{\lcm}{lcm}
\DeclareMathOperator{\Tr}{Tr}
\DeclareMathOperator{\codim}{codim}
\newcommand{\F}{\mathbb{F}}
\newcommand{\Fp}{\mathbb{F}_p}
\newcommand{\Fq}{\mathbb{F}_q}

\newcommand{\OF}{\mathcal{O}_F}
\newcommand{\OK}{\mathcal{O}_K}

\newcommand{\exes}{x_1,\cdots, x_s}
\newcommand{\Mat}{\text{Mat}}
\newcommand{\Hom}{\text{Hom}}
\newcommand{\slie}{\mathfrak{sl}}

\newcommand{\ep}{\epsilon}

\newcommand{\FN}[1][\Gamma]{F_{#1}^{\unlhd}(n)}
\newcommand{\FS}[1][\Gamma]{F_{#1}^{\leq}(n)}

\renewcommand{\mod}{\ \mathrm{mod} \ }

\usepackage[margin=2.5cm]{geometry}

\begin{document}

\title{Quantifying Residual Finiteness of Linear Groups}

\author{Daniel Franz}



\maketitle

\begin{abstract}
 Normal residual finiteness growth measures how well a finitely generated residually finite group is approximated by its finite quotients. We show that any finitely generated linear group $\Gamma\leq \GL_d(K)$ has normal residual finiteness growth asymptotically bounded above by $(n\log n)^{d^2-1}$; notably this bound depends only on the degree of linearity of $\Gamma$. If char $K=0$ or $K$ is a purely transcendental extension of a finite field, then this bound can be improved to $n^{d^2-1}$. We also give lower bounds on the normal residual finiteness growth of $\Gamma$ in the case that $\Gamma$ is a finitely generated subgroup of a Chevalley group $G$ of rank at least 2. These lower bounds agree with the computed upper bounds, providing exact asymptotics on the normal residual finiteness growth. In particular, finite index subgroups of $G(\Z)$ and $G(\Fp[t])$ have normal residual finiteness growth $n^{\dim(G)}.$ We also compute the non-normal residual finiteness growth in the above cases; for the lower bounds the exponent $\dim(G)$ is replaced by the minimal codimension of a maximal parabolic subgroup of $G$.
\end{abstract}

\section{Introduction}

Let $\Gamma$ be a finitely generated residually finite group with finite generating set $X$. If $\gamma\in \Gamma$, let $||\gamma||_X$ be the word length of $\gamma$ with respect to $X$. If $\gamma\in \Gamma$ is nontrivial, we say a finite quotient $Q$ of $\Gamma$ detects $\gamma$ if the image of $\gamma$ in $Q$ is nontrivial. Define $F^{\unlhd}_{\Gamma,X}(n)$ to be the smallest natural number $N$ such that for all $\gamma\in \Gamma$ with $||\gamma||_X\leq n$, $\gamma$ is detected by a quotient of size at most $N$.

We call the function $F^{\unlhd}_{\Gamma,X}$ the normal residual finiteness growth function of $\Gamma$. This function was first studied by Bou-Rabee in \cite{BouQrf}, and its asymptotics have been studied for virtually nilpotent linear groups \cite{BouMcRlcm}, arithmetic groups \cite{BouKal}, linear groups \cite{BouMcR}, and free groups \cite{BouApp} \cite{KasMat}, with the best current estimate for free groups given in \cite{Thom}. A related function is $F^{\leq}_{\Gamma, X}$, the non-normal residual finiteness growth function of $\Gamma$, defined as the smallest natural number $N$ such that for all $\gamma\in \Gamma$ with $||\gamma||_X\leq n$, there exists $H\leq G$ with $\gamma\not\in H$ and $[G:H]\leq N$. This function has also been studied for certain classes of groups, in particular for virtually special groups in \cite{BouHagPat} and for free groups in \cite{BouMcRlcm} \cite{Buskin} \cite{KThom}. Our goal in this paper is to obtain better estimates of the functions $F^{\unlhd}_{\Gamma,X}$ and $F^{\leq}_{\Gamma, X}$ when $\Gamma$ is a linear group.

While these functions depend on the choice $X$ of generating set, their asymptotic growths, which we call the normal residual finiteness growth of $\Gamma$ and non-normal residual finiteness growth of $\Gamma$, respectively, are independent of the choice of generating set (\cite{BouQrf}, Lemma 1.1). We thus drop the reference to $X$ for the remainder of the introduction. We compare the asymptotic growth of functions by writing $f\preceq g$ if for some $C$, $f(n)\leq Cg(Cn)$ for all $n$.

It was shown in \cite{BouMcR} that if $\Gamma$ is a finitely generated linear group over an infinite field, then $F^{\unlhd}_\Gamma(n)\preceq n^k$ for some $k$ depending on the field and the degree of linearity. A natural question is whether the dependence on the field of coefficients is necessary. Our first result is that in fact there is a uniform bound on the residual finiteness growth of finitely generated linear groups with a fixed degree of linearity.

\begin{theorem}\label{thm:GLdRF}
  Let $\Gamma\leq \GL_d(K)$ be a finitely generated linear group with $d\geq 2$.
  \begin{enumerate}[(i)]
   \item  $\FN\preceq (n\log n)^{d^2-1}$ and $\FS\preceq (n\log n)^{d-1}.$

  \item  If char $K=0$ or $K$ is a purely transcendental extension of a finite field, then \newline $\FN\preceq n^{d^2-1}$ and $\FS \preceq n^{d-1}$.
   \end{enumerate}
\end{theorem}

One potential application of normal residual finiteness growth is in showing a group is nonlinear. For a finitely generated group $\Gamma$, one can show $F_\Gamma(n)$ is super-polynomial to conclude $\Gamma$ is nonlinear. If $\Gamma$ is infinitely generated, the uniform bound of Theorem $\ref{thm:GLdRF}$ provides another method for establishing nonlinearity. In particular, this result has potential applications in the study of profinite groups.

\begin{corollary}
  Let $G$ be a group such that for each $k\in \N$, $G$ has a finitely generated subgroup $H$ with $F^{\unlhd}_H(n)\succeq n^k$. Then $G$ is nonlinear.
\end{corollary}

The proof of Theorem $\ref{thm:GLdRF}$ easily generalizes to certain algebraic groups, yielding the following more specific result. By a Chevalley group we mean a split simple group scheme defined over $\Z$, not necessarily simply connected, with irreducible root system $\Phi$. For such a group $G$, let dim$(G)$ be its dimension and $a(G)$ be the minimal codimension of a proper parabolic subgroup; these values are given in Table $\ref{tab:constants}$ and justified in Lemma $\ref{lem:ChevalleySize}$.

\begin{table}[h]
\centering
\begin{tabular}{|c|c|c|}
  \hline
  $\Phi$ & $\dim(G)$ &  $a(G)$ \\ \hline
  $A_l, l\geq 2$ &  $l^2+2l$ & $l$ \\
  $B_l, l\geq 2$ &  $2l^2+l$ & $2l-1$ \\
  $C_l, l\geq 3$ &  $2l^2+l$ & $2l-1$ \\
  $D_l, l\geq 4$ &  $2l^2-l$ & $2l-2$ \\
  $G_2$          & $14$     & $5$  \\
  $F_4$          & $52$    & $15$  \\
  $E_6$          & $78$    & $16$  \\
  $E_7$          &  $133$  & $27$  \\
  $E_8$          & $248$   & $57$  \\
  \hline
\end{tabular}
  \caption{}
  \label{tab:constants}
\end{table}

\begin{theorem}\label{thm:AlgGroupRF}
Let $G$ be an affine algebraic group scheme defined over $\Z$, let $K$ be a field, and let $\Gamma\leq G(K)$ be finitely generated.
 \begin{enumerate}[(i)]
 \item  $\FN \preceq (n\log n)^{\dim(G)}$ and, if $G$ is a Chevalley group,  $\FS\preceq (n\log n)^{a(G)}$.

\item If char $K=0$ or $K$ is a purely transcendental extension of a finite field, then \newline $\FN \preceq n^{\dim(G)}$ and, if $G$ is a Chevalley group, $\FS \preceq n^{a(G)}$.
\end{enumerate}
\end{theorem}

When $G$ is a Chevalley group of rank at least $2$, we can also determine lower bounds for residual finiteness growth, which, coupled with Theorem $\ref{thm:AlgGroupRF}$, yield precise asymptotics for residual finiteness growth. In \cite{BouKal}, Bou-Rabee and Kaletha determined the lower bound for normal residual finiteness growth of arithmetic groups of $G$. We generalize this statement to non-normal residual finiteness growth and the characteristic $p$ setting.

 \begin{theorem}\label{thm:LowerBoundRF}
If $\mathcal{O}=\Z$ or $\Fp[t]$ and $\Gamma\leq G(\mathcal{O})$ has finite index, where $G$ is a Chevalley group of rank at least 2, then $F_\Gamma^{\unlhd}(n)\succeq n^{\dim(G)}$ and $F_\Gamma^{\leq}(n)\succeq n^{a(G)}$
 \end{theorem}

Normal and non-normal residual finiteness growth can only decrease when passing to a subgroup, so Theorem $\ref{thm:LowerBoundRF}$ also gives lower bounds for all finitely generated subgroups of $G(K)$, $G$ a Chevalley group of rank at least 2 and $K$ a field. Combining this lower bound with the upper bound from Theorem $\ref{thm:AlgGroupRF}$ then gives exact asymptotics for normal and non-normal residual finiteness growth.

\begin{corollary}
  Let $G$ be a Chevalley group of rank at least 2, let $K$ be a field of characteristic 0 or a purely transcendental extension of a finite field, and let $\Gamma\leq G(K)$ be finitely generated. Put $\mathcal{O}=\Z$ if char $K=0$ and $\mathcal{O}=\Fp[t]$ if char $K=p>0$.

  If $\Gamma\cap G(\mathcal{O})\leq G(\mathcal{O})$ has finite index, then $F_\Gamma^{\unlhd}(n)\approx n^{\dim(G)}$ and $F_\Gamma^{\leq}(n)\approx n^{a(G)}$.
\end{corollary}

The main tool used to provide the uniform upper bound in characteristic 0 is a higher dimensional version of the Chebotarev density theorem formulated by Serre \cite{Serre}. In positive characteristic the needed analogue is not available, so we use an effective form of the usual Chebotarev density theorem \cite{PCheb}. Specifically, we need a statement about natural density, not Dirichlet density. Over certain fields, this causes the bounds to be powers of $n\log n$ instead of $n$. For the lower bounds we use properties of Chevalley groups and associated graded Lie algebras, as well as the congruence subgroup property.

The paper is organized as follows. After collecting some lemmas in Section 2, we prove Theorem $\ref{thm:AlgGroupRF}$ in Section 3. This is done in stages, beginning with the case of a purely transcendental extension of a finite field.  We then collect results on the Chebotarev Density Theorem which are used to prove Theorem $\ref{thm:GLdRF}$ and the remaining parts of Theorem $\ref{thm:AlgGroupRF}$ together.

In Section 4, we consider graded Lie algebras arising from Chevalley groups and relate them to the problem of finding lower bounds for normal and non-normal residual finiteness growth. The characteristic 0 part of Theorem $\ref{thm:LowerBoundRF}$ is then proved in Section 5, and in Section 6 the proof is completed in the positive characteristic setting.

\vspace{1cm}

\noindent \textit{Acknowledgments.} I would like to thank my advisor Mikhail Ershov for his great advice and support while working on this topic. I would also like to thank Martin Kassabov for discussions which greatly simplified some of the proofs of lower bounds. My thanks also to Andrei Rapinchuk for suggesting a reference that was key in proving the upper bound statement and to Khalid Bou-Rabee for providing helpful comments on an early draft of this paper.

\section{Preliminaries}

Let $\Gamma$ be a finitely generated group, generated by a finite symmetric set $X$. If $\gamma\in \Gamma$ is nontrivial, define
\begin{align*}
  D^{\unlhd}_\Gamma(\gamma)&= \min\{[\Gamma: N] : N\unlhd \Gamma, \gamma\not\in N\}, \\
  D^{\leq}_\Gamma(\gamma) &= \min\{[\Gamma: H] : H\leq \Gamma, \gamma\not\in H\}.
\end{align*}

Then the normal and non-normal residual finiteness growth of $\Gamma$ are determined by the functions
\begin{align*}
F_{\Gamma, X}^{\unlhd}(n)&=\max\{D^{\unlhd}_\Gamma(\gamma) : ||\gamma||_X\leq n, \gamma\neq 1\}, \\
F_{\Gamma, X}^{\leq}(n)&=\max\{D^{\leq}_\Gamma(\gamma) : ||\gamma||_X\leq n, \gamma\neq 1\}.
\end{align*}

We will measure asymptotic growth by writing $f\preceq g$ if there exists $C$ such that $f(n)\leq Cg(Cn)$ for all $n$. If $f\preceq g$ and $g\preceq f$ we will write $f\approx g$.

The asymptotic growths of $F_{\Gamma, X}^{\unlhd}(n)$ and $F_{\Gamma, X}^{\leq}(n)$ are independent of the generating set (Lemma 1.1, \cite{BouQrf}), so the reference to $X$ will be dropped. Another consequence of Lemma 1.1 in \cite{BouQrf} that will be used tacitly for the remainder of the paper is that $\FN[H]\preceq \FN$ and $\FS[H]\preceq \FS$ if $H\leq \Gamma$.

We will need the following result when proving lower bounds; it is contained in Lemma 2.4 in \cite{BouKal}. In particular it will allow us to pass from a Chevalley group to its simply connected cover. We include the proof for completeness.

\begin{lemma}\label{lem:surj}
  Assume $\Gamma$ and $\Delta$ are finitely generated, residually finite groups. If $f:\Gamma\to \Delta$ is a homomorphism with finite kernel, then $\FN\preceq \FN[\Delta]$ and $\FS\preceq \FS[\Delta]$.
\end{lemma}
\begin{proof}
 Since $f(\Gamma)\leq \Delta$, $\FN[f(\Gamma)] \preceq \FN[\Delta]$. Hence it suffices to show $\FN\preceq \FN[f(\Gamma)]$.

  Assume $\Gamma=\langle X\rangle$, $|X|\leq \infty$. Then $f(\Gamma)$ is generated by $f(X)=\{f(x): x\in X\}$. Since the kernel of $f$ is finite, if $n$ is sufficiently large then $f(\gamma)\neq 1$ for all $\gamma\in \Gamma$ with $||\gamma||_X=n$. Let $n$ be large enough to ensure this and let $\gamma\in \Gamma$ with $||\gamma||_X=n$. We have $||f(\gamma)||_{f(X)}\leq n$ and $f(\gamma)\neq 1$, so there exists a normal subgroup $N\unlhd f(\Gamma)$ such that $f(\gamma)\not\in N$ and $[f(\Gamma):N]\leq F^{\unlhd}_{f(\Gamma), f(X)}(n)$. Hence $N'=N\ker(f)\unlhd \Gamma$ satisfies
  \begin{equation*}
   \gamma\not\in N' \text{ and } [\Gamma:N']\leq F^{\unlhd}_{f(\Gamma), f(X)}(n),
  \end{equation*}
  so $F^{\unlhd}_{\Gamma, X}(n) \leq F^{\unlhd}_{f(\Gamma), f(X)}(n)$ and thus $\FN\preceq \FN[f(\Gamma)]$.

  The same argument with $N$ replaced by an arbitrary subgroup $H$ shows that $\FS\preceq \FS[\Delta]$.
\end{proof}

If $f\in \Fp[t][\exes]$, we treat $f$ as a polynomial with coefficients in $\Fp[t]$ and consider the degree of $f$ to be the total degree of the $x_i's$. Define the height of $f$ to be ht$(f)=\max\{\deg g(t): g(t) \mbox{ is a coefficient of } f\}.$

\begin{lemma}\label{lem:FptHom} Let $f\in \Fp[t][\exes]$ be nonzero with $\deg f\leq 2^m$. Then there exist $g_1(t),\cdots, g_s(t)\in \Fp[t]$ with $\deg g_i(t)\leq m$ for each $i$ such that $f(g_1(t), \cdots, g_s(t))\neq 0$.
\end{lemma}

\begin{proof}
We induct on $s$. Suppose $s=1$. Since $\Fp[t]$ is an integral domain and $\deg f\leq 2^m$, $f(x)$ has at most $2^m$ roots. There are at least $2^{m+1}$ elements of $\Fp[t]$ with degree at most $m$, so $f(g(t))\neq 0$ for some $g(t)$ with $\deg g(t)\leq m$.

Now assume the lemma is true for $s=n-1$ and suppose $s=n$. When considered as a polynomial over $x_s$ with coefficients in $\Fp[t][x_1,\cdots, x_{s-1}]$, $f$ has at most $2^m$ roots, so there is some $g_s(t)\in \Fp[t]$ with $\deg g(t)\leq m$ such that
$$f(x_1, \cdots, x_{s-1}, g(t))\neq 0.$$
Applying the inductive hypothesis finishes the proof.
\end{proof}

We will also need the following size estimates.

\begin{lemma}\label{lem:AlgGrpSize}
Let $G$ be an affine algebraic group scheme defined over $\Z$, $q$ a prime power. There exists a constant $C$ independent of $q$ such that $|G(\Fq)|\leq C q^{\dim(G)}$.
\end{lemma}

\begin{proof}
  Let $A$ be the Hopf algebra representing $G$, finitely generated over a field $k$, so that $G(\Fq)=\Hom_k(A,\Fq)$. By Noether normalization, $A$ is a finitely generated module over a polynomial ring $k[x_1,\cdots, x_d]$, where $d=\dim (G)$. If $A$ is generated as a module by $y_1,\cdots, y_m$, then each $y_i$ is integral over $k[x_1,\cdots, x_d]$, so for each $1\leq i\leq m$ we can find a polynomial
  $$f_i(x_1,\cdots, x_d,Y)\in k[x_1,\cdots, x_d][Y]$$
  such that $f_i(x_1,\cdots, x_d,y_i)=0$. Let
  $\displaystyle c=\max_{1\leq i\leq m} \deg f_i.$
  An element $\varphi\in \Hom_k(A,\Fq)$ is determined by the images of the $x_i$ and $y_j$. Given choices of $\varphi(x_i)$, which can be made arbitrarily, for each $1\leq j\leq m$ there are at most $c$ choices of $\varphi(y_j)$ that will satisfy $f_j(\varphi(x_1),\cdots, \varphi(x_d),\varphi(y_j))=0$. Thus $|\Hom_k(A,\Fq)|\leq c^mq^d$, so  $C=c^m$ is the required constant.
\end{proof}

\begin{lemma}\label{lem:ChevalleySize}
  Let $G$ be a Chevalley group with an embedding into $\SL_d$, $q$ be a prime power, and $H\leq G(\Fq)$ be a proper subgroup of minimal index. Then $|G(\Fq)/Z(G(\Fq))|\geq \frac{1}{2d} q^{\dim(G)}$ and $\frac{1}{2} q^{a(G)}\leq [G(\Fq): H]\leq 2 q^{a(G)}.$
\end{lemma}
\begin{proof}
  The size bound of $|G(\Fq)/Z(\Fq)|$ follows from Theorem 25, $\S$9, in \cite{St}. The index of the largest maximal subgroup of $G(\Fq)$ can be found in \cite{KL} (Theorem 5.2.2) for the classical groups, and in \cite{VGF}, \cite{VE} for the exceptional groups.
\end{proof}

\section{Upper Bounds}

\subsection{Purely transcendental extensions of finite fields}

We begin this section by proving Theorems $\ref{thm:GLdRF}$ and $\ref{thm:AlgGroupRF}$ in the case that $K$ is a purely transcendental extension of a finite field.

\begin{lemma}\label{lem:FqtHom} Let $f(t)\in \Fq[t]$ be nonzero with degree at most $n$. Then there exists a finite field $\F$ with $2n<|\F|\leq 2nq$ and a homomorphism $\phi:\Fq[t]\to \F$ such that $\phi(f(t))\neq 0$.
\end{lemma}

\begin{proof}
It is a well known generalization of a result of Gauss \cite{Roman} that the number of irreducible polynomials in $\Fq[t]$ of degree $k$ is

$$I_q(k)=\frac{1}{k}\sum_{d|k} \mu(d)q^{k/d},$$
where $\mu$ is the Mobius inversion function. It is easy to check that $kI_q(k)\geq \frac{1}{2}q^k$ for $k\geq 2$.

Given $f(t)\in\Fq[t]$, we wish to find an irreducible polynomial of appropriate degree that does not divide $f(t)$. To that end, note that if $f(t)$ is divisible by all irreducible polynomials of degree $k$, then
\[\deg f(t)\geq kI_q(k)\geq \frac{1}{2}q^k.\]

So now let $f(t)\in \Fq[t]$ have degree at most $n$. Find $M\in \N$ with $\frac{1}{2}q^{M-1}\leq n< \frac{1}{2}q^M$. Then by the above observation and since $\deg f\leq n$, there is some irreducible polynomial $h(t)$ with degree $M$ such that $h(t)$ does not divide $f(t)$. From the choice of $M$ we have
\[2n< q^M\leq 2nq,\]
 so $f(t)$ is not zero in the field $\F=\Fq[t]/(h(t))$, which satisfies $2n<|\F|\leq 2nq$.
\end{proof}

\begin{proposition}\label{prop:AlgGroupFpTransRF}
  Let $G$ be an affine algebraic group scheme defined over $\Z$ and let $K$ be a purely transcendental extension of $\Fq(t)$ for some prime power $q$. If $\Gamma\leq G(K)$ is finitely generated, then $F_{\Gamma}^{\unlhd}(n)\preceq n^{\dim(G)}$ and, if $G$ is a Chevalley group, $F^{\leq}_{\Gamma}(n)\preceq n^{a(G)}$. If $G=\GL_d$, then $F_\Gamma^{\unlhd}(n)\preceq n^{d^2-1}$ and $F_\Gamma^{\leq}(n)\preceq n^{d-1}$.
\end{proposition}

\begin{proof}
Fix an embedding $G\hookrightarrow \GL_d$, allowing us to treat elements of $\Gamma$ as invertible matrices with entries in $K$. Because $\Gamma$ is finitely generated, we may assume the transcendence basis of $K$ is finite, so write $K=\Fq(t)(\exes)$ for some indeterminates $x_i$. For notational convenience write $R=\Fq[t][\exes]$. Again using the fact that $\Gamma$ is finitely generated, $\Gamma\leq G(S)$ for some $S=R[g^{-1}]$, $g\in R$.

Let $X$ be a symmetric finite generating set of $\Gamma$. Let $m>0$ such that $g^m\gamma\in \mbox{Mat}_d(R)$ for all  $\gamma\in X$.  Now let $A\in \Gamma$ with $||A||_X=n$ and put $B=g^{mn}A\in \mbox{Mat}_d(R)$.

Since $A$ is a word of length $n$ in the elements of $X$, we may view $B$ as a word of length $n$ in the elements of $g^mX=\{g^m\gamma : \gamma\in X\}$. Let $N$ be larger than the degree or height of any entry of an element of $g^mX$.

If $A$ is not a scalar matrix, then $B$ has a nonzero off-diagonal entry or two diagonal entries with nonzero difference; in this case put $f$ equal to one of these nonzero values. We can ignore the finitely many instances where $A$ is a scalar matrix of determinant 1. If $A=a I_d$ is scalar with determinant not equal to 1, put $f=g^{mnd}(a^d-1).$ Our general strategy is to map $R[\exes]$ to an appropriately sized finite field $\F$ so that $fg$ is not mapped to 0. This map will then extend to a homomorphism $\varphi:S\to \F$ with $\varphi(f)\neq 0$, so that under the induced homomorphism
\[\varphi^*:G(S)\to G(\F),\]
 the image of $A$ is not a scalar matrix or has determinant not equal to 1.

To use our lemmas, we must first bound the degrees of the entries of $B$. Recall that $B$ can be represented as a word of length $n$ in $g^mX$, and each entry of an element of $g^mX$ has degree less than or equal to $N$. Thus each entry of $B$ has degree bounded above by $nN$; in particular, $\deg f\leq ndN$, so if we set $h=fg$, then $\deg h\leq 2ndN$ for sufficiently large $n$. Similar reasoning shows ht$(f)\leq 2ndN$.

Since $h$ is nonzero, it has some nonzero coefficient $h_0(t)\in \Fq[t]$ with $\deg h_0(t)\leq 2ndN$. By Lemma $\ref{lem:FqtHom}$, there exists a field $\F$ and homomorphism $\tau:\Fq[t]\to \F$ such that
$$2n(2dN)\leq|\F|\leq 2qn(2dN)$$
and $\tau(h_0)\neq 0$. Extending $\tau$ in the natural way to
$$\tau:\Fq[t][\exes]\to \F[\exes],$$
note that $\tau(h)\neq0$ and $\deg \tau(h)\leq 2ndN<|\F|$. Hence there exist $\alpha_1,\cdots,\alpha_s\in \F$ so that $\tau(f)(\alpha_1,\cdots, \alpha_s)\in \F^{\times}$, as is easily shown by induction on $s$.

Composing this evaluation map with $\tau$ yields a homomorphism $\theta:R\to \F$ such that $\theta(h)\neq 0$. Since the image of $\theta$ is a field and $h=fg$, $g$ is mapped to a unit by $\theta$, so $\theta$ extends to a ring homomorphism $\varphi:S\to \F$ satisfying $\varphi(f)\neq 0$. Finally, $\varphi$ induces a group homomorphism
\[\varphi^*:G(S)\to G(\F)\]
 with $\varphi^*(A)$ nontrivial;  by the choice of $f$, if det $(\varphi^*(A))=1$ then $\varphi^*(A)$ is not a scalar matrix. By Lemma $\ref{lem:AlgGrpSize}$, $|G(\F)|\leq C|\F|^{\dim(G)}$ for some constant $C$ depending only on $G$, so $|\F|\leq 4qdNn$ and $\FN\preceq n^{\dim(G)}$.

If $G$ is a Chevalley group, let $P$ be a maximal subgroup of minimal index in $G(\F)$, so $[G(\F): P]\leq 2|\F|^{a(G)}$ by Lemma $\ref{lem:ChevalleySize}$. The intersection of all conjugates of $P$ is normal, so since $G(\F)/Z(G(\F))$ is simple, this intersection is contained in $Z(G(\F))$.  But $\det A=1$, so $\varphi^*(A)$ is not a scalar matrix, and hence $\varphi^*(A)\not\in Z(G(\F))$. Thus $\varphi^*(A)$ is not in a subgroup of $G(\F)$ of index at most $2|\F|^{a(G)}$. Hence
\[A\not\in H\leq \Gamma \mbox{ and } [\Gamma:H]\leq 2 |\F|^{a(G)},\] so $\FS\preceq n^{a(G)}.$

Now assume $G=\GL_d$. If $\det(\varphi^*(A))\neq 1$, then the image of $\varphi^*(A)$ in $\F^*$ under the determinant map is nontrivial. If $\det(\varphi^*(A))=1$  then $\varphi^*(A)\in \GL_d(\F)$ is not a scalar matrix, so the image of $\varphi^*(A)$ is nontrivial in $\GL_d(\F)/Z(\GL_d(\F))$, the size of which is of order $|\F|^{d^2-1}$. Hence $F_\Gamma^{\unlhd}(n)\preceq n^{d^2-1}$.

In addition, the image of $\varphi^*(A)$ in $\GL_d(\F)/Z(\GL_d(\F))$ is in the image of $\SL_d(\F)$, which is isomorphic to $\PSL_d(\F)$. Applying the Chevalley group argument from above to $\PSL_d$ and using the fact that $[\PGL_d(\F):\PSL_d(\F)]\leq d$, we find that $F_\Gamma^{\leq}(n)\preceq n^{d-1}$.
\end{proof}

\subsection{Chebotarev density theorem}

To work with coefficients in arbitrary fields, we need to use the Chebotarev density theorem, which also plays an important role in \cite{BouQrf}, \cite{BouKal}, \cite{BouMcR}. In characteristic 0 we use the following higher dimensional generalization of the Chebotarev density theorem, a consequence of Lemma 9.3 and Theorem 9.11 in \cite{Serre}.

\begin{lemma}\label{lem:HighDimCheb}
 Let $F=\Q(\exes)$ and let $K/F$ be Galois with Galois group $G$. Let $\pi(x)$ be the number of degree 1 maximal ideals in $\Z[\exes]$ with residue field of size at most $x$, and let $\pi_1(x)$ be the number of such ideals which are unramified over $K$ such that $|\OK/\mathfrak{m}'|=|\OF/\mathfrak{m}|$ for all $\mathfrak{m'}$ lying over $\mathfrak{m}$. Then $\pi(x)\sim \dfrac{x^{s+1}}{\log(x^{s+1})}$ and $\pi_1(x)\sim \dfrac{1}{|G|}\dfrac{x^{s+1}}{\log(x^{s+1})}.$
\end{lemma}

In the characteristic $p$ case, we use an effective version of the Chebotarev density theorem for function fields. Let $L$ be a Galois extension of $\Fp(t)$ with Galois group $G$ and define $\pi(x,L/\Fp(t))$ to be the number of primes in $\Fp[t]$ of degree $x$ which are unramified and split completely in $L$. Let $P$ be the set of primes in $\Fp[t]$ which ramify over $L$, and set $D=\deg(\prod_{q\in P} q)$.
\begin{theorem}[\cite{PCheb}, Theorem 1]\label{thm:FpCheb}
Let $\F_{p^m}$ be the algebraic closure of $\Fp$ in $L$. If $m$ divides $x$, then
  \begin{equation*}
    |\pi(x,L/\Fp(t))-\frac{m}{|G|}I_p(x)|\leq \frac{p^{x/2}(2+D)}{|G|x}+D\left(1+\frac{1}{x}\right)
  \end{equation*}
\end{theorem}

\begin{lemma}\label{lem:FpCheb}
  Fix  $c_1, c_2>0$. Let $f(y)\in \Fp[t][y]$ be separable with degree $k$. If $n$ is sufficiently large, $h(t)\in\Fp[t]$ has degree at most $c_1 n\log n$, and $f(y)$ has discriminant $\Delta(f)\in \Fp[t]$ of degree less than $c_2\log n$, then there exists $c\leq 2c_1(k!)p^{k!}$, dependent on $n$, so that there exists an irreducible polynomial $g(t)\in \Fp[t]$ of degree at most $\log_p(c n\log n)$ not dividing $h(t)$ such that $f(y)$ factors into distinct linear factors mod $g(t)$.
\end{lemma}

\begin{proof}
  Let $h(t)$ and $f(y)$ satisfy the assumptions of the lemma, and let $L$ be the splitting field of $f(y)$.
Let $Q_x$ be the set of primes in $\Fp[t]$ of degree $x$ which are unramified and split completely in $L$, so $\pi(x, L/\Fp(t))=|Q_x|$. If $x>\deg(\Delta(f))$, then no element of $Q_x$ divides $\Delta(f)$, and hence $f(y)$ factors into distinct linear factors mod $q(t)$ for all $q(t)\in Q_x$.

We want to find $x$ of an appropriate size so that some $g(t)\in Q_x$ does not divide $h(t)$. To that end, observe that
  \begin{equation*}
    \deg\left(\prod_{q(t)\in Q_x} q(t)\right)=x \pi(x,L/\Fp(t)).
  \end{equation*}

  Let $\F_{p^m}$ be the algebraic closure of $\Fp$ in $L$. Since $[L:\Fp(t)]\leq k!$, $|G|\leq k!$ and $m|k!$, so if $m|x$ then Theorem $\ref{thm:FpCheb}$ yields
  \begin{align*}
    \pi(x,L/\Fp(t))&\geq \frac{m}{k!}I_p(x)-\frac{p^{x/2}(2+D)}{x|G|}-D\left(1+\frac{1}{x}\right) \\
    &\geq \frac{1}{k!}I_p(x)-\frac{p^{x/2}(2+D)}{x}-2D.
  \end{align*}

   Since $\deg(h(t))\leq c_1 n \log n$, if $x\pi(x,L/\Fp(t))>c_1n\log n$, then some $g(t)\in Q_x$ will not divide $h(t)$. Using the estimates $I_p(x)\geq \dfrac{p^x}{2x}$ and $D\leq \deg (\Delta(f))\leq c_2\log n$, we have
   \begin{align*}
    x\pi(x,L/\Fp(t)) &\geq \frac{xI_p(x)}{k!}-p^{x/2}(2+D)-2Dx \\
    &\geq \frac{p^x}{2(k!)}-p^{x/2}(2+c_2\log n)-2c_2x\log n.
  \end{align*}

   Then
  \begin{equation}\label{eq:FpCheb}
    x\pi(x,L/\Fp(t))-c_1n\log n\geq \frac{p^x}{2(k!)}-p^{x/2}(2+c_2\log n)-2c_2x\log n-c_1n\log n.
  \end{equation}
 If we set $x=\log_p(c'n \log n)$ for some $c'>0$, then the right hand side of $\eqref{eq:FpCheb}$ becomes
 \begin{equation*}
   \frac{c'n\log n}{2(k!)}-\sqrt{c'n\log n}(2+c_2\log n)-2c_2\log_p(n\log n)\log n-c_1n \log n.
 \end{equation*}
 The highest order terms in $n$ are $\dfrac{c'n\log n}{2(k!)}$ and $c_1n\log n$. Hence if $c'>2c_1(k!)$ and $n$ is sufficiently large, then the above expression is positive.

 However, we also need $x$ to be an integer divisible by $m$, while $\log_p(c' n\log n)$ may not even be an integer. Since $m$ divides $k!$, it is enough to have $k!$ divide $x$. For any $n$, the interval $(\log_p(2c_1(k!)n\log n), \log_p (2c_1(k!)p^{k!}n\log n)]$ has length $k!$, so it contains an integer multiple of $k!$. Thus there exists $c>0$ satisfying
 \begin{equation*}
  2c_1(k!)<c\leq 2c_1(k!)p^{k!}
 \end{equation*}
 such that $x=\log_p( c n\log n)\in k! \Z$. Note that while the choice of $c$ depends on $n$, its absolute value is bounded independent of $n$.

 For $n$ sufficiently large, the above choice of $c$ yields $x\pi(x,L/\Fp(t))>\deg(h(t))$, so we conclude that there is some irreducible $g(t)\in \Fp[t]$ such that $g(t)$ does not divide $h(t)$ and $f(y) \mod g(t)$ factors into distinct linear factors.
\end{proof}

\subsection{Proof of Theorems $\ref{thm:GLdRF}$ and $\ref{thm:AlgGroupRF}$} 

The primitive element theorem plays a key role in the proof of Theorem $\ref{thm:GLdRF}$; this always applies to number fields, but finite extensions of $\Fp(t)$ can be inseparable. The following lemma allows us to always work in the separable situation.

\begin{lemma}\label{lem:CharpSep}

Let $p$ be a prime and put $E_0=\Fp(\exes)$ for some $\exes$ algebraically independent over $\Fp$. If $L_0/E_0$ is a finite extension, then there is some positive integer $m$ such that if $\tilde{x}_j=x_j^{1/p^m}$ for $1\leq j\leq s$, then $L=L_0(\tilde{x}_1,\cdots, \tilde{x}_s)$ is a separable extension of $E=\F_q(\tilde{x}_1,\cdots, \tilde{x}_s)$.
\end{lemma}

\begin{proof}
First note that since $L_0/E_0$ is finite, there are some $\alpha_1,\cdots, \alpha_k\in L_0$ such that $L_0=E_0(\alpha_1,\cdots, \alpha_k)$. Each $\alpha_i$ is the root of an irreducible polynomial $f_i(y)\in E_0[y]$. In turn, each $f_i(y)=g_i(y^{p^{m_i}})$ for some irreducible, separable $g_i(y)\in E_0[y]$ and some positive integer $m_i$. Set $m=\max\{m_i\}$, put $\tilde{x}_j=x_j^{1/p^m}$ for $1\leq j\leq s$, and let $E=\F_q(\tilde{x}_1, \cdots, \tilde{x}_s).$ For each $i$, form $\tilde{g}_i(y)\in E[y]$ by replacing each $x_j$ in $g_i(y)$ by $\tilde{x}_j^{p^{m-m_i}}=x_j^{1/p^{m_i}}$.

Since we are in characteristic $p$, we then have
\begin{equation*}
  f_i(y)=g_i(y^{p^{m_i}})=\tilde{g}_i(y)^{p^{m_i}},
\end{equation*}

\noindent so $\tilde{g}_i(\alpha_i)=0$. Each $\tilde{g}_i(y)$ is separable, so $L=L_0(\tilde{x}_1, \cdots, \tilde{x}_s)$ is separable over $E$.
\end{proof}

 We now prove the remaining parts of Theorems $\ref{thm:GLdRF}$ and $\ref{thm:AlgGroupRF}$ together.

\begin{theorem}
  Let $G$ be an affine algebraic group scheme defined over $\Z$, $K$ be a field, and $\Gamma\leq G(K)$ be a finitely generated subgroup. Put $g(n)=n$ if char $K=0$ and $g(n)=n\log n$ if char $K>0$.

  Then $F^{\unlhd}_{\Gamma}(n)\preceq g(n)^{\dim(G)}$ and, if $G$ is a Chevalley group,  $F^{\leq}_\Gamma(n)\preceq g(n)^{a(G)}.$ If $G=\GL_d$, then $F^{\unlhd}_\Gamma(n)\preceq g(n)^{d^2-1}$ and $F^{\leq}_\Gamma(n)\preceq g(n)^{d-1}$.
\end{theorem}

\begin{proof}
Fix an embedding $G\hookrightarrow \GL_d$ ($G\hookrightarrow \SL_d$ if $G$ is a Chevalley group) and let $\Gamma=\langle X\rangle\leq G(K)$, where $X$ is finite and symmetric. We may assume $K$ is the field generated by the entries of the elements of $X$. We first consider the case char $K=0$, though we will see later that most of the arguments leading up to the use of the Chebotarev density theorem apply when char $K>0$.

Since $K$ is a finitely generated field, $K$ is a finite extension of $F=\Q(\exes)$ for some algebraically independent elements $\exes$. Replacing $K$ by its Galois closure if necessary, we may assume $K/F$ is Galois. By the primitive element theorem, $K=\Q(\exes)[\alpha]$ for some $\alpha\in K$, which we can choose to be integral over $\Z[\exes]=\mathcal{O}_F$. Let $f(y)\in \mathcal{O}_F[y]$ be the minimal polynomial for $\alpha$ over $\Q(\exes)$ and set $k=\deg f(y)$.

The entries of the elements of $X$ generate a ring contained in $\mathcal{O}_F[ g(\exes)^{-1} ][\alpha]=R[\alpha]$ for some $g\in \mathcal{O}_F$. Let $J$ be the ideal of $R[y]$ generated by $f(y)$. If the ring homomorphism
\begin{equation*}
  \epsilon_\alpha: F[y]\to F[\alpha]=K
\end{equation*}
is evaluation of $y$ to $\alpha$, then clearly $\ker\epsilon_\alpha$ is the ideal generated by $f(y)$. We claim the kernel of $\epsilon_\alpha|_{R[y]}$ is $J$, so that $R[\alpha]\cong R[y]/J$. This follows from the fact that $f(y)$ is monic; if some element of $R[y]$ is a multiple of $f(y)$ in $K[y]$, then it must in fact be a multiple in $R[y]$, as is seen by an easy computation of coefficients.

We now present an outline of the proof. Let $A\in \Gamma$ with $||A||_X=n$. Using the above argument, we consider $\Gamma$ as being embedded in $G(R[y]/J)$. By using appropriate coset representatives and multiplication to eliminate inverses, we examine the entries of $A$ as elements of $\OF[y]$, i.e. as polynomials with integer coefficients. We then produce a homomorphism $\OF[y]\to \Fp[y]$ under which an entry $A$ remains nontrivial and the image of $f(y)$ splits into distinct linear factors.  The end result is a homomorphism $R[y]\to \Fp$ which factors through $J$, inducing a homomorphism $G(R[y]/J)\to G(\Fp)$ which detects $A$.

Accomplishing this with no regard for the size of $G(\Fp)$ is fairly straightforward, but to achieve the desired bound, we must keep track of certain details. This is the reason we prefer to work in $\OF[y]$; these are simply polynomials with integer coefficients, with easily tracked ``size" properties.

So let $A\in \Gamma$ be nontrivial with $||A||_X= n$. If $h\in R[y]$, set $\widetilde{h}$ to be the element of $R[y]$ with $\widetilde{h}\equiv h \mod{J}$ and $\deg_y \widetilde{h}<k$. If $b=h+J\in R[\alpha]$, set $\widetilde{b}=\widetilde{h}$. For each $\gamma\in X$, let $\widetilde{\gamma}$ be the element of $\Mat_d(R[y])$ with $\widetilde{\gamma}_{ij}=\widetilde{\gamma_{ij}}$. Put $\widetilde{X}=\{\widetilde{\gamma}|\gamma\in X\}$. Let $m>0$ such that $g^m\widetilde{\gamma}\in \Mat_d(\OF[y])$ for all $\gamma\in X$. Let $N$ be the maximum degree of the entries of all the $g^m \widetilde{\gamma}$ as polynomials in $\exes$.

If $A=\gamma_1\cdots \gamma_n$, $\gamma_i\in X$, let $\widetilde{A}=\widetilde{\gamma_1}\cdots\widetilde{\gamma_n}$. Then $(g^m)^n\widetilde{A}=B$ is a product of $n$ elements chosen from $g^m\widetilde{X}$, so $B\in \Mat_d(\OF[y])$. For simplicity, suppose $B$ has a nonzero off-diagonal entry $h(\exes,y)=h$ which is not divisible by $f(y)$; the other cases can be treated as in the proof of Proposition $\ref{prop:AlgGroupFpTransRF}.$ Then for some constant $\alpha_0$ depending on $g^m$, $X$, and $s$,  we have
\begin{equation*}
  \deg_y h \leq (k-1)n,  \deg_{\exes}h\leq Nn, \mbox{ and } \mathrm{ht}(h)\leq \alpha_0^n,
\end{equation*} where ht$(h)$ is the height of $h$, the largest absolute value of a coefficient of $h$.

We want to ensure that $h$ continues to not be divisible by $f(y)$ when we evaluate the $x_i$; the easiest way to accomplish this is by degree considerations, so we now replace $h$ by $\widetilde{h}$. Since our goal is to map this element to something nonzero, it will then suffice to clear denominators and map the resulting polynomial to something nonzero. We need to do this carefully to keep track of how the $x$ degrees and coefficient sizes change.

Write
  \begin{equation*}
    f(y)=y^k+\sum_{j=0}^{k-1} a_j(\exes)y^j,
  \end{equation*}
where $a_j\in \OF$, and put $M$ to be the maximum degree of the $a_j$. Then for $r>k$, the coefficients of $\widetilde{y^r}$ will be sums of products of the $a_j$. For example,

\begin{align*}
  y^{k+1}= y\cdot y^k&\equiv y\left(-\sum_{j=0}^{k-1} a_jy^j\right) \mod{J}\\
   &= -a_{k-1} y^k-\sum_{j=0}^{k-2} a_jy^{j+1} \\
   &\equiv a_{k-1}\sum_{j=0}^{k-1} a_jy^j-\sum_{j=0}^{k-2} a_jy^{j+1} \mod{J}\\
   &= a_0a_{k-1}+\sum_{j=1}^{k-1} (a_ja_{k-1}-a_{j-1})y^j\\
   &=\widetilde{y^{k+1}}.
\end{align*}

As the above example helps illustrate, each $\widetilde{y^r}\in \OF[y]$ and the coefficients of $\widetilde{y^r}$ will include products of at most $r-(k-1)$ coefficients of $f(y)$, so $\widetilde{y^{(k-1)n}}$ includes products of at most $(k-1)(n-1)$ terms. Hence if $a(\exes)$ is a coefficient of $\widetilde{y^r}$ with $r\leq (k-1)n$, then
\begin{equation*}
  \deg a\leq M(n-1)(k-1) \mbox{ and }\mathrm{ht}(a)\leq \beta^{(n-1)(k-1)}
\end{equation*}
 for some $\beta$ independent of $n$.

We obtain $\widetilde{h}\in \OF[y]$ by replacing each $y^r$ by $\widetilde{y^r}$. Using the size and degree estimates on $\widetilde{y^r}$, we have
 \begin{equation*}
   \deg_y \widetilde{h} < k, \mbox{ } \deg_{\exes}\widetilde{h}\leq N_0n, \text{ and } \text{ht}(\widetilde{h})\leq \alpha_0^n,
 \end{equation*}
where $N_0$ and $\alpha_0$ are independent of $n$.

Viewing $\widetilde{h}$ as a polynomial with coefficients in $\OF=\Z[\exes]$, some coefficient $b(\exes)=b$ of $\widetilde{h}$ is nonzero. Let $\Delta(f(y))$ be the discriminant of $f(y)$, an element of $\OF$. Consider the polynomial
\begin{equation}\label{eq:bdef}
b'(\exes)=g(\exes)\Delta(f(y))b(\exes)\in \OF.
\end{equation}
Since the only term in this product that depends on $n$ is $b(\exes)$, $b'$ retains the properties from $\widetilde{h}$ that its degree is linear in $n,$ bounded by $cn$, and its height is exponential in $n$, bounded by $\alpha^n$.

Let $\OK$ be the integral closure of $\OF$ in $K$, and recall the definitions of $\pi(x)$ and $\pi_1(x)$ from Lemma $\ref{lem:HighDimCheb}$. We note that for each ideal $\mathfrak{m}$ counted in $\pi_1(x)$, $f(y)\mod \mathfrak{m}$ factors into linear factors.

We wish to find $\mathfrak{m}\unlhd \OF$ with $|\OF/\mathfrak{m}|\leq Cn$ for some constant $C$ independent of $n$ such that $b'(\exes)\mod \mathfrak{m}\neq 0$ and $f(y)\mod\mathfrak{m}$ factors into distinct linear factors. There are $\pi_1(Cn)$ ideals that satisfy the latter condition. We count the number that fail the first.

  Each degree one maximal ideal of $\OF=\Z[\exes]$ is of the form $(p, x_1-a_1,\cdots, x_s-a_s)$ for some $0\leq a_i\leq p-1$. For such an ideal $\mathfrak{m}$, $b'(\exes)\in \mathfrak{m}$ if and only if $b'(a_1,\cdots, a_s)\equiv 0\mod p$. Set
  \begin{equation*}
    X_p(b')=\{\textbf{a}\in \Fp^s: b'(\textbf{a})\equiv 0\mod p\}.
  \end{equation*}
   Then $|X_p(b')|=p^s$ if $p|b'$, and a straightforward induction on $s$ shows $|X_p(b')|\leq s\deg(b')p^{s-1}$ if $p$ does not divide $b'$. So if $M>0$,
  \begin{equation*}
    \sum_{p\leq M} |X_p(b')|\leq s\deg(b')\sum_{p\leq M} p^{s-1}+\sum_{p\leq M, p|b'}p^s.
  \end{equation*}

\noindent We can split the second sum into two as
\begin{equation}\label{eq:sum2}
  \sum_{p\leq M, p|b'}p^s=\sum_{p\leq \sqrt{n}, p|b'}p^s+\sum_{\sqrt{n}<p\leq M, p|b'}p^s.
\end{equation}

 Each prime in the second sum of $\eqref{eq:sum2}$ is greater than $\sqrt{n}$, so if there are $k$ terms in the sum, the product of the involved primes is at least $(\sqrt{n})^k$. Since $b'$ has height $\alpha^n$, $(\sqrt{n})^k\leq \alpha^n$, so $k\leq 2n\log\alpha/\log n$. Hence
 \begin{equation*}
   \sum_{\sqrt{n}<p\leq M, p|b'}p^s\leq \frac{2n\log\alpha}{\log n}M^s.
 \end{equation*}
Letting $n$ be sufficiently large, Lemma $\ref{lem:HighDimCheb}$ gives
\begin{equation*}
   \sum_{p\leq \sqrt{n}}p^s\leq 2 \frac{(\sqrt{n})^{s+1}}{\log ((\sqrt{n})^{s+1})},
 \end{equation*}
so if we put $M=Cn$ for some $C>1$ to be determined, we conclude that
\begin{equation*}
 \sum_{p\leq M, p|b'}p^s \leq \frac{4n\log \alpha}{\log n}M^s.
 \end{equation*}
If we let $m_0=|\text{Gal}(K/F)|$, then Lemma $\ref{lem:HighDimCheb}$ also gives
\begin{equation*}
\pi_1(M)\geq \dfrac{1}{2m_0} \dfrac{M^{s+1}}{\log (M^{s+1})} \mbox{ and } \sum_{p\leq M} p^{s-1}\leq 2 \dfrac{M^s}{\log(M^s)},
\end{equation*}
 so recalling that $\deg(b')\leq cn$, we have
 \begin{equation*}
   \sum_{p\leq M} |X_p(b')|\leq 2scn \dfrac{M^s}{\log(M^s)}+\frac{4n\log\alpha}{\log n}M^s=2nM^s\left(\frac{c}{\log M}+\frac{2\log\alpha}{\log n}\right).
 \end{equation*}

 Then we want
 \begin{equation*}
   \frac{1}{2m_0} \frac{M^{s+1}}{\log (M^{s+1})}> 2nM^s\left(\frac{c}{\log M}+\frac{2\log\alpha}{\log n}\right)\Leftrightarrow M> 4m_0(s+1)n\left(c+\frac{2\log \alpha \log M}{\log n}\right).
 \end{equation*}

 We can choose
 \begin{equation*}
 C>4m_0(s+1)\left(c+\log \alpha\left(1+\dfrac{\log C}{\log n}\right)\right)
 \end{equation*}
  independent of $n$ when $n$ is large since $\log C/\log n$ becomes arbitrarily small. Hence with $M=Cn$, we have $\pi_1(M)> \sum_{p\leq M} |X_p(b)|$. The number on the right side of this inequality is the number of degree one maximal ideals containing $b'(\exes)$, so we can in fact choose a maximal ideal $\mathfrak{m}$ of $\OF$ with $b'(\exes)\neq 0\mod \mathfrak{m}$, $f(y)\mod \mathfrak{m}$ factoring into distinct linear factors, and $\OF/\mathfrak{m}\cong \Fp$, where $p\leq M=Cn$.

  Now consider the homomorphism
  \begin{equation*}
  \psi:\OF[y]\to (\OF/\mathfrak{m})[y]\cong \Fp[y].
  \end{equation*}
   Since $\psi(b')\neq 0$ and $g|b'$, we have $\psi(g)\neq 0$, so $\psi$ extends to
   \begin{equation*}
   \pi: R[y]=\OF[g^{-1}][y]\to \Fp[y]
   \end{equation*}
    with $\pi(b')\neq 0$ and $\pi(f(y))$ a product of distinct linear polynomials.

Recalling the definition of $b'=b'(\exes)$, $\pi(b')\neq 0$ implies $\pi(\widetilde{h})\neq 0$. By our choice of $\widetilde{h}$, $\deg \pi(\widetilde{h})<\deg \pi(f)$, so $\pi(f)$ does not divide $\pi(\widetilde{h})$. In particular, $\pi(f)$ has some linear factor $y-\lambda\in \Fp[y]$ that does not divide $\pi(\widetilde{h})$. Hence under the evaluation map $\epsilon_\lambda:\Fp[y]\to \Fp$ that sends $y$ to $\lambda$, $\pi(f)$ is sent to 0 and $\pi(\widetilde{h})$ remains nontrivial. Thus we have a homomorphism
\begin{equation*}
  \epsilon_\lambda\circ\pi:R[y]\to \Fp
\end{equation*}
which maps $f(y)$ to 0 and maps $\widetilde{h}(y)$ to a nonzero element of $\Fp$. This map thus factors through $J=(f(y))\unlhd R[y]$, yielding the commutative diagram below.

\begin{center}
\hspace{0.3in} \xymatrix{
R[y] \ar[r]^\pi \ar[rd]
&\Fp[y] \ar[r]^{\epsilon_a}
&\Fp \\
&R[y]/J \ar[ru]_\varphi}
\end{center}

Recall from the beginning of the proof that $h=B_{ij}\equiv g^{mn}A_{ij}\mod{J}$ for some $i\neq j$, and that $\widetilde{h}\equiv h \mod{J}$. Then by the above diagram, the homomorphism $\varphi:R[y]/J\to \Fp$ satisfies
\begin{equation*}
0\neq \varphi(\widetilde{h}+J)=\varphi(h+J)=\varphi(g^{mn}A_{ij}+J).
\end{equation*}
 Since $g\in R^\times$, $\varphi(g^{mn}+J)\neq 0$, so we conclude that $\varphi(A_{ij}+J)\neq 0$. Thus the ring homomorphism $\varphi$ induces a group homomorphism
\begin{equation*}
  \varphi^*:G(R[y]/J)\to G(\Fp)
\end{equation*}
with $\varphi^*(A)$ a nontrivial, non-diagonal matrix. Restricting $\varphi^*$ to $\Gamma$ yields the desired homomorphism.

By the choice of $\mathfrak{m}$ we have $|\Fp|\leq Cn$, so we conclude $F^{\unlhd}_\Gamma(n)\preceq n^{\dim(G)}$.
The remaining bounds are proved as at the end of the proof of Proposition $\ref{prop:AlgGroupFpTransRF}.$

\bigskip

Now consider the case char $K=p$. Then $K$ is a finite extension of $F=\Fp(t,\exes)$ for some algebraically independent elements $t,x_1,\cdots, x_s$. By Lemma $\ref{lem:CharpSep}$ we can assume $K$ is a separable extension of $F$. As in the characteristic 0 case we may then replace $K$ by its Galois closure and assume $K/F$ is Galois. By the primitive element theorem, $K=F[\alpha]$ for some $\alpha\in K$. We again can assume $\alpha$ is integral, and we let $f(y)\in \Fp[t][\exes][y]$ be the minimal monic polynomial for $\alpha$ over $\Fp(t)(\exes)$, with $\deg f(y)=k$. In this context $\OF=\Fp[t][\exes]$ and $R=\OF[g^{-1}]$, $g\in \OF$.

One can now perform the same steps as in the characteristic 0 case, replacing $\Z$ by $\Fp[t]$ and replacing the exponential size bounds on the coefficients by linear degree bounds. Indeed, the first place where the characteristic $p$ argument diverges is just after $\eqref{eq:bdef}$. So we pick up the argument at that point, using the same notation as before.

We have a polynomial $b'(\exes)$ defined similarly as in $\eqref{eq:bdef}$,
\begin{equation*}
  b'(\exes)=g(\exes)\Delta(f(y))b(\exes)\in \OF,
\end{equation*}
\noindent with degree $m\leq c_0n$ for some $c_0$ independent of $n$. Then by Lemma $\ref{lem:FptHom}$, there exist $g_1(t), \cdots, g_s(t)$ each of degree at most $\log m$ such that $b'(g_1(t),\cdots, g_s(t))\neq 0$. If $\epsilon:\OF\to \Fp[t]$ is the evaluation homomorphism with $\epsilon(x_i)=g_i(t)$, we have $\epsilon(g)\neq 0$, so $\epsilon$ extends to $\epsilon:R\to \Fp[t]$ and thus induces a homomorphism
\begin{equation*}
  \psi:R[y]\to \Fp[t][y]
\end{equation*}
satisfying $\psi(\widetilde{h})\neq 0$, $\varphi(f)\neq 0$, and $\psi(\Delta(f))\neq 0$.

We are now in a position to use Lemma $\ref{lem:FpCheb}$. We observe that
\begin{equation*}
\deg \epsilon(b')\leq c'n+m\log m\leq c_1n\log n
\end{equation*}
 for some constants $c', c_1$. Also, the discriminant of $\psi(f(y))$ has degree at most $c_2\log n$ for some constant $c_2$ since $f(y)$ is independent of $n$ and each $g_i(t)$ has degree at most $\log (c_0n)$. Thus by Lemma $\ref{lem:FpCheb}$, we can find $c>0$ independent of $n$ and an irreducible polynomial $F(t)$ of degree less than $c n\log n$ such that $F(t)$ does not divide $\psi(b')$ and $\psi(f(y))$ factors completely mod $F(t)$.

Put $\F=\Fp[t]/(F(t))$, let $\pi_F$ be the homomorphism $\pi_F:\Fp[t][y]\to \F[y]$ induced by $\Fp[t]\to \F$, and define $\pi=\pi_F \circ \psi:R[y]\to \F[y]$. Following the same arguments as in the characteristic 0 case, one can then show the desired residual finiteness growth bounds.
\end{proof}

\section{Lower Bound Preliminaries}

For the remainder of the paper we will assume $G$ is a (simple) Chevalley group of rank at least 2, and for this section we will in addition assume $G$ is simply connected.

Let $\Phi$ be the irreducible root system of rank $l\geq 2$ associated to $G$, and let $\mathfrak{g}(\C)$ be the corresponding Lie algebra, with Chevalley basis $\{e_\alpha: \alpha\in \Phi\}\cup \{h_1,\cdots, h_l\}$. For a field $\F$, let $\mathfrak{g}(\F)$ be the Lie algebra over $\F$ with the given Chevalley basis. For background on root systems and Chevalley bases, see \cite{H}.

 Fix an embedding of $G$ in $\SL_d$. Then there is a Lie algebra embedding $\mathfrak{g}(\C)$ into $\slie_d(\C)$ such that the action of $G$ on $\mathfrak{g}(\C)$ by conjugation, via matrix multiplication, is the same as the adjoint action.

 Let $K$ be the field $\Q$ or $\Fp(t)$ with ring of integers $\mathcal{O}=\Z$ or $\Fp[t]$, respectively. , We have $G(\mathcal{O})=G(K)\cap \SL_d(\mathcal{O})$. Fix a maximal ideal $\mathfrak{m}$ and $k\in \N$. Set $R=\mathcal{O}/\mathfrak{m}^k$ and let $p$ be the characteristic of the field $\F=\mathcal{O}/\mathfrak{m}$.

Let $G_i$ be the kernel of the projection $G(R)\to G(\mathcal{O}/\mathfrak{m}^i)$ for $1\leq i\leq k$ (note that $G_k=\{1\})$. We now use the $G_i$ to construct a graded Lie algebra (see \cite{BG} or \cite{LubSeg}, Chapter 7 for more details; these kernels were also used in \cite{BouQrf} to compute normal residual finiteness growth). Each $G_i/G_{i+1}$ is an elementary abelian $p$-group for $1\leq i\leq k-1$, so we can define an $\Fp$ vector space

  \begin{equation*}
    L(G_1)=\bigoplus_{i=1}^{k-1} G_i/G_{i+1}.
  \end{equation*}

 Defining the bracket on homogeneous elements to be

 \begin{equation*}
   [xG_{i+1}, yG_{j+1}]=(x,y)G_{i+j+1},
 \end{equation*}
 where $(x,y)$ is the group commutator, gives $L(G_1)$ the structure of a Lie algebra over $\Fp$. We have a Lie algebra isomorphism
 \begin{equation*}
   L(G_1)\cong \mathfrak{g}(\F)\otimes x\F[x]/(x^k)=\bigoplus_{i=1}^{k-1} x^i \mathfrak{g}(\F).
 \end{equation*}

\noindent Now let $H\leq G(R)$. Continuing to follow \cite{BG}, define
\begin{equation*}
L(H)=\bigoplus_{i=1}^{k-1} (H\cap G_i)G_{i+1}/G_{i+1}.
\end{equation*}
Then $L(H)$ is a graded Lie subalgebra of $L(G_1)$ and the codimension of $L(H)$ in $L(G_1)$, viewed as vector spaces over $\Fp$, is $\log_p[G_1:H\cap G_1]$.

Using the realization of $L(G_1)$ as $\mathfrak{g}(\F)\otimes x\F[x]/(x^k)$, write
\begin{equation*}
L(H)=\bigoplus_{i=1}^{k-1} x^i \mathfrak{h}_i,
\end{equation*}
 where
\begin{equation*}
\mathfrak{h}_i=(H\cap G_i)G_{i+1}/G_{i+1}\cong (H\cap G_i)/(H\cap G_{i+1})
\end{equation*}
is viewed as an $\Fp$-subspace of $\mathfrak{g}(\F)$. We note that $G_1H/G_1\cong H/(H\cap G_1)$ acts on each $\mathfrak{h}_i$ by conjugation and $[\mathfrak{h}_i, \mathfrak{h}_j]\subseteq \mathfrak{h}_{i+j}$ if $i+j<k.$

In addition, we have the following result when $H$ is a normal subgroup.
\begin{lemma}\label{lem:PerfectImage}
Assume $G(R)$ is perfect and $H\unlhd G(R)$. If $H\neq G(R)$, then $G_1H\neq G(R)$.
\end{lemma}
\begin{proof}
  Recall that $R=\mathcal{O}/\mathfrak{m}^k$. For any $1\leq j\leq k-1$, there is a natural surjective homomorphism
  \[G_j/G_{j+1}\to G_jH/G_{j+1}H.\]
  If $G_jH=G(R)$, then $G(R)/G_{j+1}H$ is the image of the abelian group $G_j/G_{j+1}$. Since $G(R)$ is perfect, $G(R)/G_{j+1}H$ must be trivial, so $G_jH=G_{j+1}H$.

  In particular, if $G_1H=G(R)$, then the above argument implies $G_k H=G(R).$ Since $G_k=1$, we conclude that $H=G(R)$ if $G_1H=G(R)$.
\end{proof}

Set

\begin{equation*}
E(\F)=\bigoplus_{\alpha \in \Phi} \F e_\alpha.
\end{equation*}
We will write $E$ for $E(\F)$ when context makes clear what $\F$ is.

\begin{lemma}\label{lem:Codim}
Fix $\alpha\in \Phi$ and assume that $\alpha$ is a short root if $\Phi$ is of type $C_l, l\geq 2$.
Suppose $U,V\leq \mathfrak{g}(\F)$ satisfy $\F e_\alpha\not\subseteq [U,V]$. Then
\begin{equation*}
  \codim_E(U\cap E)+\codim_E(V\cap E)\geq 2 [\F:\Fp].
\end{equation*}
\end{lemma}

\begin{proof}
Write $[U,V]\cap \F e_\alpha=Ae_\alpha$, $A\leq \F$. Let Tr$: \F\to \Fp$ be the usual trace form. By assumption, $A$ is proper, so there exists some nonzero $a_0\in \F$ such that Tr$(a_0 a)=0$ for all $a\in A$. Replacing $U$ by $a_0 U$, we may assume Tr$(a)=0$ for all $a\in A$.

Let $B=\{b_1,\cdots, b_n\}$ be an $\F_p$-basis of $\F$, and let $B'=\{b_1', \cdots, b_n'\}$ be the dual basis of $B$ with respect to the trace, so that
\[\Tr(b_ib_j')=
\begin{cases}
  1 & \text{ if } i=j \\
  0 & \text{ if } i\neq j.
\end{cases}\]

We now construct subspaces of $E$ which intersect $U$ and $V$ trivially. Let $\beta\in \Phi$ such that $\alpha-\beta\in \Phi$ and $[e_{\alpha-\beta}, e_\beta]=e_\alpha$ (such a $\beta$ always exists because we exclude $\alpha$ from being a long root if $\Phi$ is of type $C_l$).

Set $X=\{b_ie_{\alpha-\beta}, b_i'e_\beta: 1\leq i\leq n\}$ and define an involution on $X$ by $\overline{b_i e_{\alpha-\beta}}=b_i'e_\beta.$ Then if $w_1\neq w_2\in X$, $[w_1, \overline{w_2}]=ae_\alpha \mbox{ with } \Tr(a)=0.$

 Let $X_U\subseteq X$ be maximal with respect to the property $\langle X_U\rangle \cap U=0$, and set $X_V=\{\overline{w}: w\in X\setminus X_U\}$. We show $\langle X_V\rangle \cap V=0$.

If not, there is some nonzero $v=\sum_{w\in X_V}s_w w\in V$, where each $s_w\in \Fp$. We now construct $u\in U$ such that the coefficient of $[u,v]\in \F e_\alpha$ has nonzero trace, contradicting $[U,V]\cap \F e_\alpha \subseteq Ae_\alpha$.

Some coefficient $s_{w_0}$ is nonzero, and we can assume $s_{w_0}=1$. Then $\overline{w_0}\not\in X_U$, so $u=\overline{w_0}+z\in U$ for some $z\in \langle X_U\rangle$. By the choices of $X_U$ and $X_V$, $[z, v], [\overline{w_0}, v-w_0]\in \F x_\alpha$ each have coefficients with trace 0, so for some $a\in \F$ with $\Tr(a)=0$,
\[[u,v]=[\overline{w_0}, w_0]+ax_\alpha=(\pm b_ib_i'+a) x_\alpha\in [U,V]\]
 for some $1\leq i\leq n$. But $\Tr(b_ib_i'+a)=\pm 1\neq 0$, yielding a contradiction, so $\langle X_V\rangle \cap V=0$.

Since $\langle X_U\rangle, \langle X_V\rangle \subseteq E$, we conclude that
\begin{equation*}
\codim_E(U\cap E)+\codim_E(V\cap E)\geq |X_U|+|X_V|=2[\F:\Fp].
\qedhere
\end{equation*}
\end{proof}

 For a commutative ring $R$, we write $\{x_\alpha(r): r\in R\}$ for the root subgroup of $G(R)$ corresponding to $\alpha\in \Phi$. The following result, proved in \cite{St} for the case $R$ is a Euclidean domain and in \cite{AS} when $R$ is semi-local, allows us to use a nice generating set of $G(R)$ for the rings $R$ we are interested in.

\begin{lemma}\label{lem:ElemChevalley}
  If $R$ is a Euclidean domain or a semi-local ring and $G$ is a simply connected Chevalley group, then $G(R)=\langle x_\alpha(r): \alpha\in \Phi, r\in R\rangle.$
\end{lemma}

\begin{lemma}\label{lem:ChevalleyPerfect}
  Let $G$ be a Chevalley group of type $\Phi$. Then $G(\Fp[t])$ is perfect unless $p=2$ and $\Phi$ is of type $B_2$ or $G_2$.
\end{lemma}
\begin{proof}
  The statement is proved in chapter 11 of \cite{Carter} for Chevalley groups over fields, but the same arguments apply to the polynomial ring $\Fp[t]$.
\end{proof}

If $g\in \mathcal{O}$, we write $G(\mathcal{O}, g)=\ker (G(\mathcal{O})\to G(\mathcal{O}/g)),$ and we denote the gcd of $\pi$ and $g$ as $(\pi, g)$. We call $G(\mathcal{O}, g)$ a principal congruence subgroup; any subgroup of $G(\mathcal{O})$ containing a principal congruence subgroup is called a congruence subgroup.

If the rank of $G$ is at least 2, then $G(\mathcal{O})$ has the congruence subgroup property: every finite index subgroup of $G(\mathcal{O})$ is a congruence subgroup (see Chapter 9 of \cite{PlatRap} for details). We note that it is necessary that $G$ be simply connected for this to be true.

Using the congruence subgroup property we will be able to reduce to the case of considering principal congruence subgroups. The next two statements will help us work with their images in $G(R)$.

\begin{lemma}\label{lem:CongruenceProjection}
   Let $R=\mathcal{O}/\pi^k$ for some irreducible $\pi\in \mathcal{O}$ and set $\Delta=G(\mathcal{O}, g)$ for some $g\in \mathcal{O}$. Let $\overline{\Delta}$ be the image of $\Delta$ in $G(R)$.

\begin{enumerate}[(i)]
  \item If $(\pi, g)=1$, then $\overline{\Delta}=G(R)$.
  \item     If $(\pi^k, g)=\pi^s$ with $s<k$, then $E\subseteq (\overline{\Delta}\cap G_i)G_{i+1}/G_{i+1}$ for $s\leq i\leq k-1$.
\end{enumerate}

\end{lemma}
\begin{proof}
  First assume $(\pi, g)=1$. Then for any $f\in \mathcal{O}$, there exist $h_1,h_2\in \mathcal{O}$ such that $h_1\pi^k+h_2g=f$. Thus if $\alpha\in \Phi$,
  \begin{equation*}
  x_\alpha(h_2 g)=x_\alpha(f)x_\alpha(-h_1 \pi^k)\in \Delta,
  \end{equation*}
  so $x_\alpha(f \mod \pi^k)\in \overline{\Delta}.$ Lemma $\ref{lem:ElemChevalley}$ then implies $\overline{\Delta}=G(R)$.

  Now assume $(\pi^k, g)=\pi^s$ with $s<k$. Then using similar reasoning as above, for any $\alpha\in \Phi$ and any $f\in \mathcal{O}$, $x_\alpha(\pi^s f \mod \pi^k)\in \overline{\Delta}$. Hence for $s\leq i\leq k-1$, \begin{equation*}
  \{x_\alpha(\pi^i f \mod \pi^k):\alpha\in \Phi, f\in \mathcal{O}\}\subseteq \overline{\Delta}\cap G_i,
  \end{equation*}
   proving the lemma.
\end{proof}

 We note that the statement in $(ii)$ of the above lemma is not optimal. In fact $(\overline{\Delta}\cap G_i)G_{i+1}/G_{i+1}=G(R)$ in this case, but we only require the weaker statement to prove the following corollary.

\begin{corollary}\label{cor:FiniteIndexCodim}
  With the same setup as in Lemma $\ref{lem:CongruenceProjection}$, let $H\leq \overline{\Delta}$ and fix $\alpha\in \Phi$, a short root if $\Phi$ is type $C_l$. Assume $(\pi^k, g)=\pi^s$ and $\F e_\alpha \not\subseteq \mathfrak{h_j}$ for some $1\leq j\leq k-1$ such that $s<j/2$.  If $s=0$, then
  \[\codim_{L(\overline{\Delta})} L(H)\geq [\F:\Fp](j-1).\]
   If $s\geq 1$, then
  \begin{equation*}
    \codim_{L(\overline{\Delta})} L(H)\geq [\F:\Fp](j-2s+1).
  \end{equation*}
\end{corollary}

\begin{proof}
  Since $\F e_\alpha\not\subseteq\mathfrak{h}_j$ and $[\mathfrak{h}_i, \mathfrak{h}_{j-i}]\subseteq \mathfrak{h}_j$ for $1\leq i\leq j-1$, we have
  $\F e_\alpha\not\subseteq [\mathfrak{h}_i, \mathfrak{h}_{j-i}].$ Put $\mathfrak{d}_i=(\overline{\Delta}\cap G_i)G_{i+1}/G_{i+1}$.

  If $s=0$, then $\overline{\Delta}=G(R)$ by Lemma $\ref{lem:CongruenceProjection}$, so $E\subseteq \mathfrak{d}_i$ for $1\leq i\leq k-1$. Then  Lemma $\ref{lem:Codim}$ implies
  \begin{equation*}
  \codim_{\mathfrak{d}_i} (\mathfrak{h}_i)+\codim_{\mathfrak{d}_j} (\mathfrak{h}_{j-i})\geq 2[\F:\Fp]
  \end{equation*}
  for $1\leq i\leq j-1$. Hence
  \begin{equation*}
  \codim_{L(\overline{\Delta})} L(H)\geq [\F:\Fp](j-1).
  \end{equation*}

  If $s\geq 1$, Lemma $\ref{lem:CongruenceProjection}$ gives that $E\subseteq \mathfrak{d}_i$ for $s\leq i\leq k-1$. There are $j-2s+1$ integers in the interval $[s, j-s]$, so the previous reasoning yields the desired inequality.
\end{proof}

\begin{lemma}\label{lem:InvariantIdeal}
  Let $\F$ be a finite field of characteristic $p$ and $G$ a simply connected Chevalley group. For all but finitely many $p$, the adjoint action of $G(\F)$ on $\mathfrak{g}(\F)$ is irreducible. The exceptions are given in Table $\ref{tab:InvariantIdeals}$, along with the largest possible dimension of a proper ideal $I\subseteq \mathfrak{g}(\F)$ invariant under the action of $G(\F)$ in those cases. If $G$ is type $B_2$ and $p=2$, then any invariant ideal $I$ is either the center or contains $\F e_\alpha$ for all short roots $\alpha$.
\end{lemma}
\begin{proof}
  See Theorem 2.1 in \cite{HOG}.
\end{proof}

\begin{table}[h]
\centering
\begin{tabular}{|c|c|c|c|}
  \hline
  $\Phi$ & $p$ &  max $\dim(I)$ & $\min \codim(I)$\\ \hline
  $A_l, l\geq 2$ &  $p|(l+1)$ & $1$ &$l^2+2l-1$ \\
  $B_l, l\geq 3$ &  $2$ & $2l+2$ & $2l^2-l-2$\\
  $C_l, l\geq 2$ &  $2$ & $2l^2-l$ & $2l$\\
  $D_l, l\geq 4$ &  $2$ & $2$ & $2l^2-l-2$\\
  $G_2$          & $3$     & $7$ & $7$ \\
  $F_4$          & $2$    & $26$ & $26$ \\
  $E_6$          & $3$    & $1$ & $77$ \\
  $E_7$          &  $2$  & $1$ & $132$ \\
  \hline
\end{tabular}
  \caption{}
  \label{tab:InvariantIdeals}
\end{table}

The final lemma of this section enables us to apply Lemma $\ref{lem:InvariantIdeal}$ to the situation where we consider $\mathfrak{g}(\F)$ as a vector space over $\Fp$. Because the proof is technical, we postpone it to the end of the paper.

\begin{lemma}\label{lem:InvariantSubspace}
  Let $\F$ be a finite field of characteristic $p$ such that $|\F|\geq 4$, and let $G$ be a simply connected Chevalley group with root system $\Phi$. Let $V$ be a proper $\Fp$-subspace of $\mathfrak{g}(\F)$. If $V$ is $G(\F)$-invariant, then $\F V$, the $\F$ subspace spanned by $V$, is a proper ideal of $\mathfrak{g}(\F)$ which is invariant under the action of $G(\F)$.
\end{lemma}

\section{Lower bounds in characteristic 0}
We continue with the notation of the previous section, with $G$ remaining a simply connected Chevalley group with a fixed embedding into $\SL_d$. Fix $\alpha\in \Phi$, a short root if $G$ is type $C_l$. We first provide lower bounds for the normal and non-normal residual finiteness growth of $G(\Z)$. The values of $\dim(G)$ and $a(G)$ can be found in Table $\ref{tab:constants}$.

\begin{lemma}\label{lem:Char0BaseCase}
  Let $R=\Z/p^k$ for a prime $p$, $k\geq 1$. Let $\Delta=G(\Z,N)$ and $\overline{\Delta}$ be the image of $\Delta$ in $G(R)$. Assume $(p^k, N)=p^s$. Let $r\geq N$ be sufficiently large and set
  \begin{align*}
  L_r&=(\lcm(1,2,\cdots, r))^{3(\dim(G)+s)}, \\
  M_r&=x_\alpha(L_r \mod p^k).
  \end{align*}
   If $M_r\not\in H\leq \overline{\Delta}$, then $[\overline{\Delta}:H]\geq \dfrac{1}{2}r^{a(G)}$. If in addition $H\unlhd \overline{\Delta}$, then $[\overline{\Delta}:H]\geq \dfrac{1}{2d} r^{\dim(G)}$.
\end{lemma}
\begin{proof}

  Let $M_r$, $H$ be as in the statement and suppose $p^{m-1}||L_r$, by which we mean $p^{m-1}$ is the largest power of $p$ dividing $L_r$. We have $m\leq k$ since $M_r\neq 1$ and $M_r\in \overline{\Delta}$  since $N|L_r$. The argument splits into a few cases. We will consider $H$ as an arbitrary subgroup and as a normal subgroup in each case.

  \medskip

\noindent \textbf{Case 1:} $k=1$. Since $k=1$, we have $R\cong\Fp$ and $p>r\geq N$, so $(p, N)=1$. By Lemma $\ref{lem:CongruenceProjection}$, $\overline{\Delta}=G(\Fp)$, so $H$ is a proper subgroup of $G(\Fp)$. Then by Lemma $\ref{lem:ChevalleySize}$, $[G(\Fp):H]\geq \frac{1}{2}p^{a(G)}.$ Since $M_r$ is nontrivial, $p$ does not divide $L_r$, so by construction of $L_r$, $p>r$. Hence $[G(\Fp):H]\geq \dfrac{1}{2}r^{a(G)}$, as desired.

  If in addition $H$ is normal, then $H\subseteq Z(G(\Fp))$ since $G(\Fp)/Z(G(\Fp))$ is simple. Thus by Lemma $\ref{lem:ChevalleySize}$,
  \begin{equation*}
  [G(\Fp):H]\geq |G(\Fp)/Z(G(\Fp))|\geq \frac{1}{2d} p^{\dim(G)}>\frac{1}{2d}r^{\dim(G)}.
  \end{equation*}

\medskip

\noindent \textbf{Case 2:} $k\geq2$, $m=1$. Let $G_1$ be the kernel of the projection $G(R)\to G(\Fp),$ and recall the graded Lie algebras $L(G_1)$ and $L(H)$ defined in the previous section. Since $m=1$,  $p$ does not divide $L_r$, so again $p>r$ and $\overline{\Delta}=G(\Fp)$. We also have $M_r\not\in G_1$. If in addition $G_1H\neq G(R)$, then the image of $H$ in $G(R)/G_1\cong G(\Fp)$ is proper, so
\[[G(R):H]\geq \frac{1}{2}p^{a(G)}>\frac{1}{2}r^{a(G)}.\]
 If $H$ is normal, then by the same reasoning as before we see that $[G(R):H]> \dfrac{1}{2d}r^{\dim(G)}.$

If $G_1 H=G(R)$, then since $p>r$ is large, $G_1H/G_1\cong G(\Fp)$ acts irreducibly on $\mathfrak{g}(\Fp)$ by Lemma $\ref{lem:InvariantIdeal}$, so each $\mathfrak{h}_j$ is trivial or all of $\mathfrak{g}(\Fp)$. If all are $\mathfrak{g}(\Fp)$, this forces $H=G(R)$, contradicting $M_r\not\in H$. Thus $\mathfrak{h}_j$ is trivial for some $j$ and
\begin{equation*}
\codim_{L(G_1)}L(H)\geq \codim \mathfrak{h}_j=\dim(G),
\end{equation*}
so $[G(R):H]\geq p^{\dim(G)}>r^{\dim(G)}$.

\medskip

  \noindent \textbf{Case 3:} $k\geq 2$, $m\geq 2$.
   Since $M_r\not\in H$ and $p^{m-1}||L_r$, we have $M_r\in G_{m-1}\setminus G_m,$ so $\Fp e_\alpha\not\subseteq \mathfrak{h}_{m-1}$. If $p^l||\lcm(1,\cdots, r)$, then
   \begin{equation*}
   m-1=3(\dim(G)+s)l \mbox{ and } p^{(l+1)\dim(G)}>r^{\dim(G)}.
   \end{equation*}
    In particular, $s< j/2$, so by Corollary $\ref{cor:FiniteIndexCodim}$, if $s\geq 1$ then
   \begin{equation*}
   \codim_{L(\overline{\Delta})}(L(H))\geq m-2s.
   \end{equation*}
   Since
   \begin{equation*}m-2s=3(\dim(G)+s)l-2s+1\geq \dim(G)(l+1),
   \end{equation*}
   we conclude that
   \begin{equation*}
   [\overline{\Delta}:H]\geq[\overline{\Delta}\cap G_1:H\cap G_1]\geq p^{\dim(G)(l+1)}>r^{\dim(G)}.
   \end{equation*}
   A similar argument works when $s=0$, using the corresponding inequality from Corollary $\ref{cor:FiniteIndexCodim}$.
\end{proof}

\begin{theorem}\label{thm:Char0LowerBoundRF}
    Let $G$ be a Chevalley group of rank at least 2, not necessarily simply connected, and let $\Delta$ be a finite index subgroup of $G(\Z)$. Then $\FN[\Delta]\succeq n^{\dim(G)}$ and $\FS[\Delta]\succeq n^{a(G)}$.
\end{theorem}

\begin{proof}
 Let $G_{sc}$ be the simply connected cover of $G$. Then the natural map $G_{sc}(\Z)\to G(\Z)$ has finite kernel, so by Lemma 2.4 in \cite{BouKal}, the residual finiteness growth of $G(\Z)$ is bounded below by that of $G_{sc}(\Z)$. Thus we may assume $G$ is simply connected, with irreducible root system $\Phi$.

 Then $G(\Z)$ satisfies the congruence subgroup property, so $\Delta$ contains some principal congruence subgroup. Since residual finiteness growth can only decrease by passing to a subgroup, we may assume $\Delta=G(\Z, N)$ for some $N\in \Z$. Let $s$ be the largest power of a prime dividing $N$.

 Fix $r\geq N$ sufficiently large and put $L_r=(\lcm(1,2,\cdots, r))^{3(\dim(G)+s)}$. Fix some $\alpha\in \Phi$, a short root if $G$ is type $C_l$. We show that $M_r=x_\alpha(L_r)$ is in every subgroup of $G(R)$ of sufficiently small index. First we need to determine the word length of $M_r$ in $\Delta$.

 By Theorem $A$ in \cite{LMR}, there exists a generating set $X$ of $G(R)$ so that
 \begin{equation*}
 ||M_r||_X\leq C_1 \log |L_r|
 \end{equation*}
  for some $C_1>0$. By the Prime Number Theorem, $\lcm(1,\cdots, r)\sim e^r$, so $\log|L_r|\leq C_2(d+s)r$ for some absolute constant $C_2$. Since $G(Z)$ is quasi-isometric to $\Delta$, we conclude that
  \begin{equation*}
  ||M_r||_Y\leq C r
  \end{equation*}
   for some generating set $Y$ of $\Delta$ and some constant $C$ independent of $r$.

 Now suppose $M_r\not\in H\leq \Delta$. By the congruence subgroup property of $G(\Z)$, $H\supseteq G(\Z, N')$ for some $N'\in \Z$. Let $R=\Z/N'$ and let $N'=\prod_{i=1}^k p_i^{k_i}$ be the prime factorization of $N'$. Write $G_{(i)}=G(\Z/p_i^{k_i})$ for each $i$, so that
 \begin{equation*}
 G(R)\cong \prod_{i=1}^{k} G_{(i)}.
 \end{equation*}
 Let $\pi_{N'}$ be the natural projection $G(\Z)\to G(R)$. Then $\pi_{N'}(M_r)\not\in \pi_{N'}(H)$, so in some $G_{(i)}$, $\overline{M_r}=x_\alpha(L_R \mod p_i^{k_i})\not\in \overline{H}$, where these are the images in $G_{(i)}$.  So by Lemma $\ref{lem:Char0BaseCase}$,
 \begin{equation*}
 [\Delta: H]\geq [\overline{\Delta}, \overline{H}]\geq \dfrac{1}{2} r^{a(G)},
 \end{equation*}
  and $[\Delta: H]\geq \dfrac{1}{2d} r^{\dim(G)}$ if $H\unlhd \Delta$. Recalling that $M_r$ has word length $n\leq Cr$ finishes the argument.
\end{proof}

\section{Lower bounds in characteristic $p$}

We continue using the same setup as in the previous section but now deal with the groups $G(\Fp[t])\subseteq \SL_d(\Fp[t])$. We first prove the following lemma allowing us to handle the case when $G$ is type $B_2$ and $p=2$. Let the root system of type $B_2$ have roots $\{\pm \ep_1, \pm \ep_2, \pm(\ep_1\pm \ep_2)\}$, as usual.

\begin{lemma}\label{lem:B2Case}
  Fix $k\geq 1$ and set $R=\F_2[t]/f(t)^k$ for some irreducible $f(t)$. Let $G$ be a simply connected Chevalley group of type $B_2$ and let $H$ be a proper normal subgroup of $G(R)$. If $G(R)'\not\subseteq H$, then $G_1 H\neq G(R)$.
\end{lemma}

\begin{proof}
 We first set up some notation to make the computations more clear. For $1\leq j\leq k$, put $G(j)=G(\F_2[t]/\pi^j)$, and for $1\leq i\leq j$, set $G(j)_i=\ker(G(j)\to G(i)).$ We will continue writing $G_i$ for $G(k)_i$. Note that $G(k)=G(R)$.

    We first show $G_{k-1} H\neq G(k).$ Assume otherwise. Then
  \begin{equation*}
  G(k)/H=G_{k-1}H/H\cong G_{k-1}/(H\cap G_{k-1})
  \end{equation*}
  is a nontrivial abelian quotient of $G(k)$, so $[G(k),G(k)]\subseteq H$, a contradiction.

  Recall that we can view $\mathfrak{h}_j=(H\cap G_j)G_{j+1}/G_{j+1}$ as a subspace of $\mathfrak{g}(\F)$.

  We now assume for the sake of contradiction that $G_1 H=G(k)$. Since $G_{k-1}H\neq G(k)$, there exists some $2\leq j\leq k-1$ such that $G_{j-1} H=G(k)$ and $G_{j}H\neq G(k)$. Then
  \begin{equation*}
  \mathfrak{g}(\F)/\mathfrak{h}_{j-1}\cong G_{j-1}/(H\cap G_{j-1})G_{j}\cong G_{j-1} H/G_{j} H
  \end{equation*}
  is nontrivial, so $\mathfrak{h}_{j-1}\neq \mathfrak{g}(\F)$. We now show that $G_1H=G(k)$ also implies that $\mathfrak{h}_{j-1}=\mathfrak{g}(\F).$

  Put $H(j)=G_{j}H/G_{j}$ and observe $H(j)$ is properly contained in $G(j)$. Since $G_{j-1} H=G(k)$,  $G(j)_{j-1} H(j)=G(j),$ so $G(j)'\subseteq H(j)$. Hence
  \begin{equation*}
  x_{\ep_1}(\pi^{j-1})x_{\ep_2+\ep_1}(\pi^{j-1})=(x_{\ep_2}(1), x_{\ep_1-\ep_2}(\pi^{j-1}))\in H(j)\cap G(j)_{j-1},
  \end{equation*}
  so $e_{\ep_1}+e_{\ep_2+\ep_1}\in \mathfrak{h}_{j-1}.$

  The subspace $\mathfrak{h}_{j-1}$ is invariant under the action of $G_1H/G_1=G(\F)$, so by Lemma $\ref{lem:InvariantSubspace}$, $\F \mathfrak{h}_{j-1}$ is a proper ideal of $\mathfrak{g}(\F)$. Then by Lemma $\ref{lem:InvariantIdeal}$, $\F \mathfrak{h}_{j-1}$ is the center of $\mathfrak{g}(\F)$ or contains $\F e_\alpha$ for each short root $\alpha$. Clearly $\F\mathfrak{h}_{j-1}$ is not the center, so it contains $e_{\ep_1}$ and thus also contains $e_{\ep_2+\ep_1}$. This then forces $\F\mathfrak{h}_{j-1}$ to be all of $\mathfrak{g}(\F)$, a contradiction.
\end{proof}

Let $\alpha\in \Phi$ be a short root if $G$ is type $C_l$.

\begin{lemma}\label{lem:CharpBaseCase}
  Let $R=\Fp[t]/f(t)^k$ for an irreducible polynomial $f(t)$, $k\geq 1$. Let $\Delta=\ker(G(\Fp[t]), g(t))$ and let $\overline{\Delta}$ be the image of $\Delta$ in $G(R)$. Assume $(f(t)^k, g(t))=f(t)^s$. Fix $r\geq \deg(g(t))$, and set
  \begin{equation*}
   L_r(t)=(\lcm\{h(t)\in \Fp[t] : \deg (h(t))\leq r\})^{3(\dim(G)+s)}.
  \end{equation*}
   If $p=2$ and $G$ is type $C_2,$ let
  \begin{equation*}
   M_r=x_{\ep_1}(L_r(t) \mod f(t)^k)x_{\ep_1+\ep_2}(L_r(t) \mod f(t)^k),
  \end{equation*}
  and otherwise set
  \begin{equation*}
   M_r=x_\alpha(L_r(t) \mod f(t)^k).
  \end{equation*}
   If $M_r\not\in H\leq \overline{\Delta}$, then $[\overline{\Delta}:H]\geq \dfrac{1}{2}p^{ra(G)}$. If in addition $H\unlhd \overline{\Delta}$, then $[\overline{\Delta}: H]\geq \dfrac{1}{2d}p^{r\dim(G)}$.
\end{lemma}
\begin{proof}
  Let $M_r$, $H$ be as in the statement, put $q=p^{\deg(f(t))}$, and suppose $f(t)^{m-1}||L_r(t)$, where $m\leq k$ since $M_r\neq 1$. Observe that $M_r\in \overline{\Delta}$ since $g(t)|L_r(t)$. The argument splits into a few cases. As in the proof of Lemma $\ref{lem:Char0BaseCase}$, we will treat $H$ as an arbitrary subgroup and then as a normal subgroup in each case. The arguments are very similar to the characteristic 0 case, so details will sometimes be skipped.

  \medskip

\noindent \textbf{Case 1:} $k=1$.   Since $k=1$, we have $R\cong \Fq$ and $\deg(f(t))>r\geq \deg(g(t))$, so $f(t)$ and $g(t)$ are relatively prime. Then $\overline{\Delta}=G(\Fq)$ by Lemma $\ref{lem:CongruenceProjection}$, so $H$ is a proper subgroup of $G(\Fq)$. By Lemma $\ref{lem:ChevalleySize}$, $[G(\Fq):H]\geq \dfrac{1}{2}q^{a(G)}.$ Hence
\begin{equation*}
 [\overline{\Delta}: H]=[G(\Fq):H]\geq \dfrac{1}{2}p^{ra(G)}.
\end{equation*}

If $H$ is normal, then $H\subseteq Z(G(\Fq))$ since $G(\Fq)/Z(G(\Fq))$ is simple, so Lemma $\ref{lem:ChevalleySize}$ gives
\begin{equation*}
 \overline{\Delta}:H]\geq \frac{1}{2d}q^{\dim(G)}>\frac{1}{2d}p^{r\dim(G)}.
\end{equation*}

\medskip

\noindent \textbf{Case 2:} $k\geq 2, m=1$.   Since $m=1$, we again have $\deg(f(t))>r$ and $\overline{\Delta}=G(R)$. Let $G_1$ be the kernel of the projection $G(R)\to G(\Fq),$ and recall the graded Lie algebras $L(G_1)$ and $L(H)$.

   We first consider the case $H\unlhd G(R)$. If $G(R)$ is perfect, then by Lemma $\ref{lem:PerfectImage}$, $G_1H\neq G(R)$. Hence the image of $H$ in $G(\Fq)$ is proper and
   \begin{equation*}
    [\overline{\Delta}: H]\geq \dfrac{1}{2d} p^{r\dim(G)}
   \end{equation*}
    as before. Otherwise, by Lemma $\ref{lem:ChevalleyPerfect}$,  $p=2$ and $G$ is type $B_2$ or $G_2$. In the former case,
    \begin{equation*}
     M_r=[x_{\ep_1}(1), x_{\ep_1-\ep_2}(L_r(t) \mod f(t)^k)]\in G(R)',
    \end{equation*}
    so $G(R)'\not\subseteq H$ and thus $G_1H\neq G(R)$ by Lemma $\ref{lem:B2Case}$, yielding the desired bound as shown above. If $G$ is type $G_2$ and $G_1H=G(R)$, then $G_1 H/G_1\cong G(\Fq)$ acts irreducibly on each $\mathfrak{h}_i$ by Lemmas $\ref{lem:InvariantIdeal}$ and $\ref{lem:InvariantSubspace}$. Hence some $\mathfrak{h}_i$ is trivial, so
    \begin{equation*}
     \codim_{L(G_1)} L(H)\geq \dim(G) \deg(f(t))
    \end{equation*}
     and
    \begin{equation*}
     [\overline{\Delta}:H]\geq p^{\dim(G) \deg(f(t))}\geq p^{r\dim(G)}.
    \end{equation*}

   If $H$ is an arbitrary subgroup of $G(R)$, then the case $G_1H\neq G(R)$ again reduces to a previous argument. So assume $G_1 H=G(R)$. Then $G_1 H/G_1 \cong G(\Fq)$ acts on each $\mathfrak{h}_i$, so for each $i,$ either $\mathfrak{h}_i=\mathfrak{g}(\F)$ or $\F \mathfrak{h}_i$ is a proper ideal, using Lemma $\ref{lem:InvariantSubspace}$. Since $\mathfrak{h}_i\subseteq \F \mathfrak{h}_i$, by examining Table $\ref{tab:InvariantIdeals}$ and Table $\ref{tab:constants}$ we see that each $\mathfrak{h_i}$ is all of $\mathfrak{g}(\Fq)$ or has codimension at least $a(G)\deg(f(t))$. Since $H$ is proper, not all the $\mathfrak{h_i}$ can be $\mathfrak{g}(\Fq)$, so \begin{equation*}
   \codim_{L(G_1)}L(H)\geq a(G)\deg(f(t))>ra(G).
   \end{equation*}
    Thus $[G(R):H]\geq p^{ra(G)}$.

    \medskip

\noindent \textbf{Case 3:} $k\geq 2, m\geq 2$.    We handle $H$ being normal and arbitrary simultaneously.  Since $M_r\not\in H$ and $f(t)^{m-1}||L_r(t)$, we have $M_r\in G_{m-1}\setminus G_m,$ so $\Fq e_\alpha\not\subseteq \mathfrak{h}_j$ for some $m-1\leq j\leq k-1$ ($\Fq(e_{\ep_1}+e_{\ep_1+\ep_2})\not\subseteq \mathfrak{h}_j$ if $G$ is type $B_2$, $p=2$).

 By the construction of $L_r(t)$, $f(t)^{m-1}||L_r(t)$ implies $m-1=3(\dim(G)+s)l$ for some integer $l\geq 1$ satisfying $\deg(f(t))(l+1)>r$. In particular, $s<j/2$, so by Corollary $\ref{cor:FiniteIndexCodim}$, if $s\geq 1$ then
  \begin{align*}
  \codim_{L(\overline{\Delta})}(L(H))&\geq \deg(f(t))(j-2s+1) \\
  &\geq\deg(f(t))(m-2s).
  \end{align*}
  We have
  \begin{equation*}
   m-2s=3(\dim(G)+s)l-2s+1\geq \dim(G)(l+1),
  \end{equation*}
 so
 \begin{equation*}
  \deg(f(t))(m-2s)\geq \dim(G)\deg(f(t))(l+1)>r\dim(G),
 \end{equation*}
 and hence
 \begin{equation*}
 [\overline{\Delta}:H]\geq[\overline{\Delta}\cap G_1:H\cap G_1]\geq p^{r\dim(G)}.
 \end{equation*}
  A similar argument works when $s=0$, using the corresponding inequality from Corollary $\ref{cor:FiniteIndexCodim}$.

 We note that while Corollary $\ref{cor:FiniteIndexCodim}$ does not directly apply in the case $G$ is type $B_2$, $p=2$, the same arguments in Lemma $\ref{lem:Codim}$ work when using $e_{\ep_1}+e_{\ep_1+\ep_2}$ in place of $e_\alpha$ because
 \begin{equation*}
  e_{\ep_1}+e_{\ep_1+\ep_2}=[e_{\ep_1}+e_{\ep_2}, e_{\ep_1-\ep_2}+e_{\ep_2-\ep_1}].
  \qedhere
 \end{equation*}
\end{proof}

\begin{theorem}\label{thm:CharpLowerBoundRF}
       Let $G$ be a Chevalley group, not necessarily simply connected, of rank at least 2, let $p$ be a prime, and let $\Delta$ be a finite index subgroup of $G(\Fp[t])$. Then $\FN[\Delta]\succeq n^{\dim(G)}$ and $\FS[\Delta]\succeq n^{a(G)}$.
\end{theorem}

\begin{proof}
As in the proof of Theorem $\ref{thm:Char0LowerBoundRF}$, we may assume $G$ is simply connected and $\Delta=G(\Fp[t], g(t))$ for some $g(t)\in \Fp[t]$. Let $s$ be the largest power of an irreducible polynomial dividing $g(t)$.

 Fix $r\geq \deg(g(t))$ and set
 \begin{equation*}
  L_r(t)=(\lcm\{h(t): \deg(h(t))\leq r\})^{3(\dim(G)+s)}.
 \end{equation*}
   Let $\Phi$ be the root system of $G$, and let $\alpha\in \Phi$, with the extra condition that $\alpha$ is a short root if $\Phi$ is of type $C_l, l\geq 2$. Set
   \begin{equation*}
      M_r=\begin{cases}
        x_{\ep_1}(L_r(t)) x_{\ep_1+\ep_2}(L_r(t)) & \text{ if } \Phi=C_2, p=2 \\
        x_\alpha(L_r(t)) & \text{ otherwise }
      \end{cases}.
    \end{equation*}
    By Theorem $A$ in \cite{LMR}, there exists a generating set $X$ of $G(\Fp[t])$ so that
   \begin{equation*}
   ||M_r||_X\leq C_1\deg(L_r(t))
   \end{equation*}
    for some constant $C_1$. The degree of $\lcm\{h(t)\in \Fp[t] : \deg (h(t))\leq r\}$ is at most $2p^r$, so $\deg(L_r(t))\leq 6(\dim(G)+s)p^r$. Hence $||M_r||_X\leq C_2p^r$ for some constant $C_2.$ Since $G(\Fp[t])$ is quasi-isometric to $\Delta$, we conclude that $M_r$ has word length $n\leq Cp^r$ for some constant $C$ with respect to some generating set of $\Delta$.

 The remaining argument is the same as in the proof of Theorem $\ref{thm:Char0LowerBoundRF}$. Substituting Lemma $\ref{lem:CharpBaseCase}$ for Lemma $\ref{lem:Char0BaseCase}$, one shows that if $M_r\not\in H\leq \Delta$, then $[\Delta: H]\geq \dfrac{1}{2} p^{ra(G)}$, and if $H$ is normal then $[\Delta: H]\geq \dfrac{1}{2d} p^{r\dim(G)}.$
 \end{proof}

\appendix
 \section{Proof of Lemma $\ref{lem:InvariantSubspace}$}


 Here we give the postponed proof of Lemma $\ref{lem:InvariantSubspace}.$ We will use the descriptions of the irreducible root systems given in section 12 of \cite{H} except for $G_2$; in this case we fix a base $\{\alpha_S, \alpha_L\}$, where $\alpha_S$ and $\alpha_L$ are short and long roots, respectively, so that the short roots are $\{\pm \alpha_S, \pm_(\alpha_S+\alpha_L), \pm (2\alpha_S+\alpha_L)\}$ and the long roots are $\{\pm \alpha_L, \pm(3\alpha_S+\alpha_L), \pm(3\alpha_S+2\alpha_L)\}$.

 \begin{lemma}\label{lem:LongRoots}
  Let $\Phi$ be an irreducible root system.
  \begin{enumerate}[(1)]
	\item If $\alpha,\gamma\in \Phi$ and $\alpha$ is a long root, then $\gamma-2\alpha\in \Phi$ if and only if $\gamma=\alpha$.
	\item If $\Phi$ is not of type $C_l, l\geq 2$, then there exist long roots $\alpha,\beta\in\Phi$ such that $\alpha+\beta\in \Phi$ and $\alpha-\beta\not\in \Phi$.
  \end{enumerate}
\end{lemma}

\begin{proof}
We first prove $(1)$. Let $\alpha,\gamma\in\Phi$ with $\alpha$ long. If $\gamma=\alpha$ then $\gamma-2\alpha=-\alpha\in \Phi$. So assume that $\gamma-2\alpha\in \Phi$. If $\theta$ is the angle between $\alpha$ and $\gamma$, then $(\alpha, \gamma)=|\alpha||\gamma|\cos \theta$, so $(\alpha,\gamma)\leq |\alpha||\gamma|$, with equality if and only if $\gamma=\alpha$. Then
\begin{align*}
  |\gamma-2\alpha|^2 &= (\gamma-2\alpha, \gamma-2\alpha) \\
  &= |\gamma|^2-4(\alpha, \gamma)+4|\alpha|^2 \\
  &\geq |\gamma|^2-4|\alpha||\gamma|+4|\alpha|^2 \\
  &=(|\gamma|-2|\alpha|)^2,
\end{align*}
with equality if and only if $\gamma=\alpha$. But $(|\gamma|-2|\alpha|)^2\geq |\alpha|^2$ and $\alpha$ is a long root, so $|\gamma-2\alpha|^2 > (|\gamma|-2|\alpha|)^2$ is not possible since $\gamma-2\alpha\in \Phi$. Hence $\gamma=\alpha$.

  We now prove $(2)$ case by case. If $\Phi$ is a simply laced root system, then there are no root strings of length greater than $2$, so any choice of $\alpha,\beta\in \Phi$ with $\alpha+\beta\in \Phi$ will suffice.

  If $\Phi$ is of type $B_l$, $l\geq 3,$ or $F_4$, set $\alpha=\ep_1-\ep_2$ and $\beta=\ep_2-\ep_3$. Then $\alpha+\beta=\ep_1-\ep_3\in \Phi$ and $\alpha-\beta\not\in\Phi$.

  If $\Phi$ is of type $G_2$, then put $\alpha=\alpha_L$ and $\beta=3\alpha_S+\alpha_L$. Then $\alpha+\beta=3\alpha_S+2\alpha_L\in \Phi$ and $\alpha-\beta=-3\alpha_S\not\in \Phi$.
\end{proof}

Let $\Phi$ be an irreducible root system, let $G$ be a Chevalley group of type $\Phi$ and let $\mathfrak{g}$ be the Lie algebra of type $\Phi$ with Chevalley basis $B=\{e_\alpha: \alpha\in \Phi\}\cup\{h_1,\cdots, h_l\}$. The following equations and more information on Chevalley groups can be found in \cite{Carter}. 

Let $\mathfrak{g}(\Z)$ be the $\Z$-span of $B$; this is a Lie algebra over $\Z$. If $\alpha\in \Phi$ and $v\in \mathfrak{g}(\Z)$, then by the properties of Chevalley bases, $\frac{1}{2} [e_\alpha, [e_\alpha, v]]$ and $\frac{1}{6}[e_\alpha,[e_\alpha,[e_\alpha, v]]]$ are both in $\mathfrak{g}(\Z)$. If $p$ is a prime, then using the natural map $\mathfrak{g}(\Z)\to \mathfrak{g}(\Fp)$, we can interpret these expressions as elements in $\mathfrak{g}(\Fp)$, regardless of the choice of $p$. In particular, these expressions make sense even if $p=2$ or $p=3$.

Using this interpretation, if $\F$ is a field, $\alpha\in \Phi$, and $t\in \F$, then 
  \begin{equation}\label{eq:action0}
    x_\alpha(t) \cdot v=v+t[e_\alpha, v]+t^2\frac{1}{2}[e_\alpha,[e_\alpha, v]]+t^3 \frac{1}{6}[e_\alpha,[e_\alpha,[e_\alpha, v]]],
  \end{equation}
  where the final term is always 0 if $\Phi$ is not of type $G_2$.
  
We will need the following specific instances of $\eqref{eq:action0}$. 
\begin{align*}
  x_\alpha(t) \cdot e_\alpha &= e_\alpha, \\
  x_\alpha(t) \cdot e_{-\alpha} &=e_{-\alpha}+th_{\alpha}-t^2e_\alpha, \\
  x_\alpha(t) \cdot h_\alpha &= h_\alpha-2te_\alpha.
\end{align*}
If $\alpha,\beta\in \Phi$ are linearly independent, i.e. $\beta\neq \pm \alpha$, then
\begin{align*}
  x_\alpha(t)\cdot h_\beta &= h_\beta- \langle \alpha, \beta\rangle e_\alpha, \\
  x_\alpha(t) \cdot e_\beta &= e_\beta+\sum_{i=1}^q M_{\alpha, \beta, i} t^i e_{i\alpha+\beta},
\end{align*}
where $M_{\alpha, \beta, i}\in \{\pm 1, \pm 2, \pm 3\}$.

We are now ready to prove Lemma $\ref{lem:InvariantSubspace}$.

\newtheorem*{lem:InvariantSubspace}{Lemma \ref{lem:InvariantSubspace}}
\begin{lem:InvariantSubspace}
  Let $\F$ be a finite field of characteristic $p$ such that $|\F|\geq 4$, and let $G$ be a simply connected Chevalley group with root system $\Phi$. Let $V$ be a proper $\Fp$-subspace of $\mathfrak{g}(\F)$. If $V$ is $G(\F)$-invariant, then $\F V$, the $\F$-subspace spanned by $V$, is a proper ideal of $\mathfrak{g}(\F)$ which is invariant under the action of $G(\F)$.
\end{lem:InvariantSubspace}

\begin{proof}
 Let $\Phi$ have rank $l$ and fix a Chevalley basis $B=\{e_\alpha: \alpha\in \Phi\}\cup\{h_1,\cdots, h_l\}$ of $\mathfrak{g}(\F)$. The $\F$-subspace $\F V$ is an ideal of $\mathfrak{g}(\F)$ if $[\mathfrak{g}(\F), \F V]\subseteq \F V$, but it is sufficient to check that $[e_\alpha, \F V]\subseteq \F V$ for all $\alpha\in \Phi$, as we now show.

  Let $\Pi=\{\alpha_1,\cdots, \alpha_l\}$ be the base for $\Phi$ associated to the Chevalley basis $B$, so that $h_i=[e_{\alpha_i}, e_{-\alpha_i}]$ for $1\leq i\leq l$. Then using the Jacobi identity, for $v\in \F V$ and $1\leq i\leq l$ we have
  \begin{equation*}
	[h_i, v]=[[e_{\alpha_i}, e_{-\alpha_i}], v]=[e_{\alpha_i}, [e_{-\alpha_i}, v]]-[e_{-\alpha_i}, [e_{\alpha_i}, v]].
  \end{equation*}
  Thus if $[e_\alpha, \F V]\subseteq \F V$ for all $\alpha\in \Phi$, then $[h_i, \F V]\subseteq \F V$ as well, so $[\mathfrak{g}(\F), \F V]\subseteq \F V$ and $\F V$ is an ideal. We now proceed to the proof of the lemma.

  First assume that $V$ is actually an $\F$-subspace of $\mathfrak{g}(\F)$, so $V=\F V$. If $\alpha\in \Phi$, $\lambda\in \F$, and $v\in V$, then we can write $\eqref{eq:action0}$ as
  \begin{equation}\label{eq:action}
    x_\alpha(\lambda) \cdot v-v=\lambda[e_\alpha, v]+\lambda^2\frac{1}{2}[e_\alpha,[e_\alpha, v]]+\lambda^3 \frac{1}{6}[e_\alpha,[e_\alpha,[e_\alpha, v]]]\in V.
  \end{equation}
  Since $|\F|\geq 4$, there exist three distinct nonzero elements $s,t,u\in \F$. Fix $v\in V$ and $\alpha\in \Phi$ and write the right hand side of $\eqref{eq:action}$ as $\lambda z_1 +\lambda^2 z_2+\lambda^3 z_3\in V$. Since $V$ is an $\F$-subspace, this implies $z_1+\lambda z_2+\lambda^2 z_3\in V$. Using $s,t,u$ in place of $\lambda$, we have
  \begin{align*}
    v_1 &= z_1+sz_2+s^2 z_3 \in V, \\
    v_2 &= z_1+tz_2+t^2 z_3 \in V, \\
    v_3 &= z_1+uz_2+u^2 z_3 \in V.
  \end{align*}
  The matrix of this linear system is Vandermonde and hence invertible. Then $z_1=[e_\alpha, v]$ is a linear combination of $v_1, v_2,$ and $v_3$, and thus $[e_\alpha, v]\in V$.

  Therefore $V$ is an ideal of $\mathfrak{g}(\F)$. Since $V$ is assumed to be proper, $V=\F V$ is a proper ideal of $\mathfrak{g}(\F)$, so the lemma is proved in this case.

  We now assume for the remainder of the proof that $V$ is an $\Fp$-subspace, but not necessarily an $\F$-subspace, of $\mathfrak{g}(\F)$. Then $\F V$ is a $G(\F)$-invariant $\F$-subspace of $\mathfrak{g}(\F)$ and thus an ideal by the above argument. It remains to show that $\F V\neq \mathfrak{g}(\F)$. We split the proof into three cases which cover different restrictions on $p$ and $\Phi$.
  
  \medskip
  
\noindent \textbf{Case 1:} $p\neq 2$, $\Phi$ not of type $G_2$.  If $s\in \F,$ $\alpha\in \Phi$, and $v\in V$, then by $\eqref{eq:action0}$ we have
  \begin{equation*}
   x_\alpha(s)\cdot v-x_\alpha(-s)\cdot v=2s[e_\alpha, v]\in V,
  \end{equation*}
  so $s [e_\alpha, v]\in V$. Thus $s[h_i, v]\in V$ for all $s\in \F$, $1\leq i\leq l$, $v\in V$ by the argument at the beginning of this proof.

  Therefore the $\F$-span of $\{[x,v]: x\in \mathfrak{g}(\F), v\in V\}$ is contained in $V$ and hence is not equal to $\mathfrak{g}(\F)$. But this set is just $[\mathfrak{g}(\F), \F V]$. Since char $\F\neq 2$, $[\mathfrak{g}(\F), \mathfrak{g}(\F)]=\mathfrak{g}(\F)$, so we must have $\F V\neq \mathfrak{g}(\F)$.

\medskip

\noindent \textbf{Case 2:} $p$ any prime, $\Phi$ not of type $C_l$, $l\geq 2$. Assume $\F V=\mathfrak{g}(\F)$; we will show that this implies $V=\mathfrak{g}(\F)$, a contradiction.

 Let $E_L$ and $E_S$ be the $\F$-subspaces of $\mathfrak{g}(\F)$ spanned by $\{e_\alpha: \alpha \text{ long}\}$ and $\{e_\alpha: \alpha \text{ short}\}$, respectively, so $\mathfrak{g}(\F)=H\oplus E_S \oplus E_L$, with the convention that $E_S=0$ if $\Phi$ is simply laced. Since $V$ is $G(\F)$-invariant and $x_\alpha(t)\cdot e_{-\alpha}=e_{-\alpha}+th_\alpha-t^2 e_\alpha$ for $\alpha\in \Phi$, $t\in \F$, to show that $V=\mathfrak{g}(\F)$ it suffices to show $E_S\oplus E_L\subseteq V$.

Fix $v\in V$, which we write as
  \begin{equation}\label{eq:v}
    v=h+\sum_{\beta\in \Phi} s_\beta e_\beta\in V,
  \end{equation}
   where $h\in H$ and $s_\beta\in \F$. If $\gamma\in \Phi$ is a long root and $\delta\in \Phi$, then $2\gamma+\delta\in \Phi$ if and only if $\delta=-\gamma$ by Lemma $\ref{lem:LongRoots} (1)$, so for any $t\in \F$,
  \begin{equation}\label{eq:LongAction}
    x_\gamma(t)\cdot v-v=t [e_\gamma, v]- t^2s_{-\gamma} e_\gamma\in V.
  \end{equation}

  By Lemma $\ref{lem:LongRoots}(2)$, we can find long roots $\alpha$ and $\beta$ such that $\alpha+\beta\in \Phi$ and $\alpha-\beta\not\in \Phi$. If $\gamma\in \Phi$, then by Lemma $\ref{lem:LongRoots}(1)$, $\gamma-2\alpha\in \Phi$ if and only if $\gamma=\alpha$.
  
  By assumption, $\F V=\mathfrak{g}(\F)$, so there exists $v\in V$ with $s_\alpha\neq 0$. Hence to show that $\F e_{-\alpha}\subseteq V$, it is sufficient to prove the following claim.

  \medskip
  
  \noindent \textbf{Claim:} Let $\gamma_1=\beta$, $\gamma_2=-\alpha$, and $\gamma_3=-(\alpha+\beta)$. 
  For any $t\in \F$ and any $v\in V$ written as in $\eqref{eq:v}$,
  \begin{equation*}
    t[ e_{\gamma_3},[e_{\gamma_2}, [e_{\gamma_1},v]-s_{-\gamma_1}e_{\gamma_1}]]=\pm t s_\alpha e_{-\alpha}\in V.
  \end{equation*}
\begin{proof}
Fix $t\in \F$, $v\in V$. Set $v_1=[e_{\gamma_1},v]-s_{-\gamma_1}e_{\gamma_1},$ which is in $V$ by $\eqref{eq:LongAction}$. Since
\begin{equation*}
 -(\gamma_1+\gamma_2)=\alpha-\beta\not\in \Phi,
\end{equation*}
 the coefficient of $e_{-\gamma_2}$ in $v_1$ is 0, and thus $v_2=[e_{\gamma_2}, v_1]=[e_{\gamma_2},[e_{\gamma_1}, v]]\in V$ by $\eqref{eq:LongAction}$. In addition, $v_2\in E_S\oplus E_L$. Similarly, the coefficient of $e_{-\gamma_3}$ in $v_2$ is 0 since $-(\gamma_1+\gamma_2+\gamma_3)=2\alpha\not\in \Phi$, so $v_3=t[e_{\gamma_3}, v_2]\in V$ by $\eqref{eq:LongAction}$ and $v_3\in E_S\oplus E_L$.

For any $\gamma\in \Phi$,
\begin{equation*}
\gamma+\gamma_1+\gamma_2+\gamma_3=\gamma-2\alpha,
\end{equation*}
and $\gamma-2\alpha\in \Phi$ if and only if $\gamma=\alpha$. We also have $\gamma_2+\gamma_3=\beta-2\alpha\not\in\Phi$ and $\gamma_1+\gamma_2+\gamma_3=-2\alpha\not\in \Phi$, so in fact $v_3=\pm t s_\alpha e_{-\alpha}$ as claimed.
\end{proof}

The Weyl group $W \leq G(\F)$ of $\Phi$ acts transitively on $\{e_\gamma : \gamma \text{ long}\}$, so since $\F e_{-\alpha}\subseteq V$ and $\alpha$ is a long root, we conclude that $E_L\subseteq V$. If $\Phi$ is simply laced, this immediately implies $ V=\mathfrak{g}(\F)$, the desired contradiction.

If $\Phi$ is of type $B_l, l\geq 3$ or $F_4$, then for $t\in \F$,
\begin{equation*}
  x_{\ep_1}(t)\cdot e_{\ep_2-\ep_1}-e_{\ep_2-\ep_1}=\pm t e_{\ep_2}\pm t^2 e_{\ep_1+\ep_2}\in V.
\end{equation*}
Since $\ep_1+\ep_2$ is a long root and $E_L\subseteq V$, we have $\F e_{\ep_2} \subseteq V$. By the transitive action of $W$ on $\{e_\gamma: \gamma \text{ short}\}$, $E_S\subseteq V$ and hence $V=\mathfrak{g}(\F)$.

If $\Phi$ is of type $G_2$, then for $t\in \F$,
\begin{equation*}
  x_{-\alpha_S-\alpha_L}(t)\cdot e_{\alpha_L}-e_{\alpha_L}=\pm t e_{-\alpha_S},
\end{equation*}
so $\F e_{-\alpha_S}\subseteq V$. Hence $E_S\subseteq V$ and $V=\mathfrak{g}(\F)$.

\medskip

\noindent \textbf{Case 3:} $p=2$, $\Phi$ of type $C_l$, $l\geq 2$. We again assume $\F V=\mathfrak{g}(\F)$ for the sake of contradiction. Let $\gamma_1=2\ep_2$ and $\gamma_2=-2\ep_1$. These are long roots with $\gamma_1+\gamma_2\not\in \Phi$ and $\gamma+\gamma_1+\gamma_2\in\Phi$ if and only if $\gamma=\ep_1-\ep_2$. Then by the same reasoning as in the argument for Case 2, if $t\in \F$ and $v\in V$ is written as in $\eqref{eq:v},$ we have
  \begin{equation*}
  t[e_{\gamma_2}, [e_{\gamma_1}, v]-s_{-\gamma_1}e_{\gamma_1}]=\pm t s_{\ep_1-\ep_2} e_{\ep_2-\ep_1}\in V.
  \end{equation*}
  Since $\F V=\mathfrak{g}(\F)$, $s_{\ep_1-\ep_2}\neq 0$ for some $v\in V$, so $\F e_{\ep_2-\ep_1}\subseteq V$ and hence $E_S\subseteq V$.

   To show $E_L\subseteq V$, let $v\in V$ with $s_{2\ep_2}\neq 0$. Since $E_S\subseteq V$, we can assume $v$ is of the form $v=h+\sum_{\alpha \text{ long}} s_\alpha e_\alpha$. The only long roots $\alpha$ satisfying $\ep_1-\ep_2+\alpha\in \Phi$ are $\alpha=-2\ep_1$ and $\alpha=2\ep_2$, and
   \begin{align*}
	 x_{\ep_1-\ep_2}(1)\cdot e_{-2\ep_1}-e_{-2\ep_1} &= \pm e_{-\ep_1-\ep_2} \pm e_{2\ep_1}, \\
	 x_{\ep_1-\ep_2}(1)\cdot e_{2\ep_2}-e_{2\ep_2} &= \pm e_{\ep_1+\ep_2} \pm e_{-2\ep_2}.
   \end{align*}
   Therefore
   \begin{equation*}
	 x_{\ep_1-\ep_2}(1)\cdot v -v = s e_{\ep_1-\ep_2} \pm s_{-2\ep_1} e_{-\ep_1-\ep_2} \pm s_{-2\ep_1} e_{2\ep_1} \pm s_{2\ep_2} e_{\ep_1+\ep_2} \pm s_{2\ep_2} e_{-2\ep_2} \in V,
   \end{equation*}
   where $[e_{\ep_1-\ep_2}, h]=s e_{\ep_1-\ep_2}$ for some $s\in \F$.
   Again using the fact that $E_S\subseteq V$, we conclude that
  \begin{equation*}
  v_1=\pm s_{2\ep_2} e_{2\ep_1}\pm s_{-2\ep_1} e_{-2\ep_2}\in V.
  \end{equation*}
  Then if $t\in \F$,
  \begin{equation*}
	x_{e_{-\ep_1-\ep_2}}(t)\cdot v_1 -v_1 = \pm ts_{2\ep_2} e_{\ep_1-\ep_2}\pm t^2 s_{2\ep_2} e_{-2\ep_2}\in V,
  \end{equation*}
  so $t^2 s_{2\ep_2}e_{-2\ep_2}\in V.$ But $\F$ is a finite field with characteristic 2, so $\F^2=\F$, and thus we conclude that $E_L\subseteq V$ and hence $V=\mathfrak{g}(\F)$, a contradiction.
\end{proof}


\begin{thebibliography}{99}

\bibitem{AS} E. Abe and K. Suzuki, \textit{On normal subgroups of Chevalley groups over commutative rings,} Tohoku Math. J. \textbf{28} (1976) no. 1, 185-198.

\bibitem{BG} Y. Barnea and R. Guralnick, \textit{Subgroup growth in some pro-$p$ groups}, Proceedings of the AMS. \textbf{130} (2001), 653-659.

\bibitem{BouQrf} K. Bou-Rabee, \textit{Quantifying residual finiteness}, J. of Algebra \textbf{323} (2010), 729-737.

\bibitem{BouApp} K. Bou-Rabee, \textit{Approximating a group by its solvable quotients}, N.Y.J. of Math \textbf{17} (2011), 699-712.

\bibitem{BouHagPat} Bou-Rabee, Hagen, Patel, \textit{Residual finiteness growths of virtually special groups}, Math. Z. \textbf{279} (2015) no. 1-2, 297-310.

\bibitem{BouKal} K. Bou-Rabee and T. Kaletha, \textit{Quantifying residual finiteness of arithmetic groups}, Compos. Math. \textbf{148} (2012), 907-920.

\bibitem{BouMcRlcm} K. Bou-Rabee and D.B. McReynolds, \textit{Asymptotic growth and least common multiples in groups}, Bull. Lond. Math. Soc. \textbf{43} (2011), 1059-1068.

\bibitem{BouMcR} K. Bou-Rabee and D.B. McReynolds, \textit{Extremal behavior of divisibility functions}, Geometriae Dedicata, \textbf{175} (2015), 407-415.

\bibitem{Buskin} N. Buskin, \textit{Economical separability in free groups}, Sib. Math. J. \textbf{50} (2009) no. 4, 603-608.

\bibitem{Carter} R. Carter, \textit{Simple Groups of Lie Type}, Pure Appl. Math., vol 28, Wiley, London (1972).

\bibitem{HOG} G. M. D. Hogeweij, \textit{Almost classical Lie algebras: I, II,} Indag. Math. \textbf{44} (1982) no. 4, 441-460.

\bibitem{H} J. Humphreys, \textit{Introduction to Lie Algebras and Representation Theory}, Springer, New York (1972).

\bibitem{KasMat} M. Kassabov and F. Matucci, \textit{Bounding the residual finiteness of free groups,} Proc. Am. Math. Soc. \textbf{139} (2011), 2281-2286.

\bibitem{KL} P. Kleidman and M. Liebeck, \textit{The subgroup structure of the finite classical groups}, Cambridge University Press, 1990.

\bibitem{KThom} G. Kozma and A. Thom, \textit{Divisibility and laws in finite simple groups}, Math. Ann. \textbf{364} (2016) no. 1-2, 79-95.

\bibitem{LMR} Lubotzky, Mozes, Raghunathan, \textit{The word and Riemannian metrics of semisimple groups}, Publ. Math. Inst. Hautes Etudes Sci. \textbf{91} (2000), 5-53.

\bibitem{LubSeg}  A. Lubotzky and D. Segal, \textit{Subgroup Growth}, Progress in Mathematics, 212.
Birkh¨auser Verlag, Basel, 2003.

\bibitem{PlatRap} V. Platonov and A. Rapinchuk, \textit{Algebraic groups and number theory.} Translated from the 1991 Russian original by Rachel Rowen. Pure and Applied Mathematics, 139. Academic Press, Inc., Boston, MA, 1994.

\bibitem{Roman} S. Roman, \textit{Field Theory}, Springer-Verlag, 1995.

\bibitem{Serre} J.-P. Serre, \textit{Lectures on $N_x(p)$}, Research Notes in Mathematics 11, CRC Press, 2012.

\bibitem{St} R. Steinberg, \textit{Lectures on Chevalley groups,} Yale University, 1968.

\bibitem{Thom} A. Thom, \textit{About the length of laws for finite groups}, http://arxiv.org/abs/1508.07730.

\bibitem{PCheb} V. Murty and J. Scherk, \textit{Effective versions of the Chebotarev density theorem for function fields,} C. R. Acad. Sci. Paris S\'{e}r. I Math. \textbf{319} (1994), 523-528.

\bibitem{VGF} A. V. Vasilyev, \textit{Minimal permutation representations of finite exceptional groups of types $G_2$ and $F_4$}, Algebra and Logic. \textbf{35} (1996) no. 6, 371-383.

\bibitem{VE} A. V. Vasilyev, \textit{Minimal permutation representations of finite exceptional groups of types $E_6,$ $E_7$, and $E_8$}, Algebra and Logic. \textbf{36} (1997) no. 5, 302-310.


\end{thebibliography}
\end{document}